\documentclass[11pt]{amsart}
\usepackage[utf8]{inputenc}
\usepackage{amsmath,amssymb,mathtools,amsthm}
\usepackage{amsfonts}
\usepackage{enumerate}
\usepackage{graphicx}
\usepackage{fancyhdr}
\usepackage[mathscr]{euscript}
\usepackage{mathrsfs}
\usepackage{caption}
\usepackage{xcolor}
\usepackage{dsfont}
\usepackage{pstricks}
\usepackage{graphicx}
\usepackage[all]{xy}
\usepackage{comment}

\definecolor{bleu_sombre}{rgb}{0,0,0.6}  \definecolor{rouge_sombre}{rgb}{0.8,0,0}\definecolor{vert_sombre}{rgb}{0,0.6,0}
\usepackage[plainpages=false,colorlinks,linkcolor=bleu_sombre,
citecolor=rouge_sombre,urlcolor=vert_sombre,breaklinks]{hyperref}

\usepackage{cleveref}

\makeatletter

\@addtoreset{equation}{section}
\makeatother

\theoremstyle{plain}
\newtheorem{definition}{Definition}[section]
\newtheorem{theorem}[definition]{Theorem}
\newtheorem{remark}[definition]{Remark}
\newtheorem{corollary}[definition]{Corollary}
\newtheorem{proposition}[definition]{Proposition}
\newtheorem{lemma}[definition]{Lemma}

\newcommand{\dd}{\mathrm d}
\newcommand{\R}{\mathbf R}
\newcommand{\C}{\mathbf C}
\newcommand{\Z}{\mathbf Z}
\newcommand{\T}{\mathbf T}
\newcommand{\Op}{\mathsf{Op}_h^{B}}
\newcommand{\Opu}{\mathsf{Op}_1^{B}}

\begin{document}

\title{Quantum unique ergodicity for magnetic Laplacians on $\T^2$}

\author[L.~Morin]{L\'eo Morin}
\address{Department of Mathematics, University of Copenhagen, Universitetsparken 5, DK-2100 Copenhagen \O, Denmark}
\email{lpdm@math.ku.dk}

\author[G.~Rivi\`ere]{Gabriel Rivi\`ere}
\address{Laboratoire de Math\'ematiques Jean Leray, Universit\'e de Nantes, UMR CNRS 6629, 2 rue de la Houssini\`ere, 44322 Nantes Cedex 03, France}
\address{Institut Universitaire de France, Paris, France}
\email{gabriel.riviere@univ-nantes.fr}

\begin{abstract}
 Given a smooth integral two-form and a smooth potential on the flat torus of dimension $2$, we study the high energy properties of the corresponding magnetic Schr\"odinger operator. Under a geometric control condition on the magnetic field, we show that every sequence of high energy eigenfunctions equidistributes. Equivalently, the magnetic Laplacian on the torus satisfies the quantum unique ergodicity property even if the Liouville measure is not ergodic for the underlying classical flow (the Euclidean geodesic flow on the $2$-torus).
\end{abstract}

\maketitle

\section{Introduction}

Let $\T^2=\R^2/\Z^2$ be the standard Euclidean torus of dimension $2$ that we endow with a \emph{fixed} smooth and real-valued $2$-form $B(x_1,x_2)\dd x_1\wedge \dd x_2$ satisfying
\begin{equation}\label{e:quantization}
\widehat{B}_0:=\int_{\T^2} B(x_1,x_2)\dd x_1\wedge \dd x_2\in 2\pi\Z.
\end{equation}
One can define a magnetic (or Bochner) Laplacian\footnote{In fact, up to gauge choices, there is a one parameter family $\mathscr{L}_\alpha$ of such operators, where $\alpha$ varies in $\R^2/(2\pi\Z^2)$. Yet, our main result is the same for each parameter $\alpha$ (and for each gauge choice) and we omit it in the introduction.} $\mathscr{L}$ on $\T^2$. The magnetic Laplacian on $\R^2$ is 
$$
\mathscr{L}:=(-i\partial_{x_1}-A_1)^2+(-i\partial_{x_2}-A_2)^2,
$$
for some $(A_1,A_2)$ verifying $B=\partial_{x_1}A_2-\partial_{x_2} A_1.$ When viewed as an operator on the torus, $\mathscr{L}$ acts on functions $u\in\mathscr{C}^\infty(\R^2,\C)$ verifying the magnetic-periodicity condition
$$
\forall (x_1,x_2)\in\R^2,\ u(x_1,x_2)=e^{-\frac{i\widehat{B}_0}{2} x_1} u(x_1,x_2-1)=e^{\frac{i\widehat{B}_0}{2} x_2}u(x_1-1,x_2),
$$
provided~\eqref{e:quantization} holds true.
Such functions are denoted by $\mathscr{C}^\infty(\T^2,L)$ and we refer to Section~\ref{s:laplacian} for a precise definition and statements about this self-adjoint, Laplace type operator.

 Given $V\in\mathscr{C}^\infty(\T^2,\R)$, we are interested in solutions to the eigenvalue problem,
\begin{equation}\label{e:eigenvalue}
\left(\mathscr{L}+V\right)u_\lambda =\lambda^2 u_\lambda,\quad \| u_\lambda\|_{L^2}=1,
\end{equation}
as $\lambda\rightarrow +\infty$. In order to state our main result, we introduce the following function on $S^*\mathbf{T}^2\simeq \T^2\times\mathbf{S}^1$,
$$
B_\infty(x,\theta):=\lim_{T\rightarrow +\infty}\frac{1}{T} \int_0^T B(x+t\theta)\dd t,\quad (x,\theta)\in S^*\T^2.
$$
Recall that, if $k\cdot\theta \neq 0$ for every $k\in\Z^2\setminus\{0\}$, then $B_\infty(x,\theta)=\widehat{B}_0$ while for rational values of $\theta$, the function $B_\infty$ may depend in a nontrivial (but smooth) way on $x\in\mathbf{T}^2$. When $B_\infty>0$ everywhere, one often says that the magnetic field verifies a \emph{geometric control condition}~\cite{RauchTaylor1974, BardosLebeauRauch1992}. Our main theorem states that solutions to~\eqref{e:eigenvalue} enjoy a \emph{quantum unique ergodicity} property under this geometric assumption.

\begin{theorem}\label{t:maintheo} Suppose that $B_\infty(x,\theta)>0$ for every $(x,\theta)\in S^*\T^2$ and that $B,V\in\mathscr{C}^\infty(\T^2,\R)$. Then, for any sequence of solutions $(u_\lambda)_{\lambda\rightarrow +\infty}$ to~\eqref{e:eigenvalue} and for any $a\in \mathscr{S}^0(T^*\T^2)$, one has\footnote{Here, $\dd\theta$ denotes the Euclidean volume measure on the circle $\mathbf{S}^1=\{\theta\in\R^2:|\theta|=1\}$ and $\dd x$ the Lebesgue measure on $\T^2$.}
$$
\lim_{\lambda\rightarrow +\infty}\langle \mathsf{Op}_{\lambda^{-1}}^B(a)u_\lambda,u_\lambda\rangle =\frac{1}{2\pi} \int_{\T^2\times\mathbf{S}^1}a(x,\theta) \dd x\dd\theta,
$$
where $\mathsf{Op}_{\lambda^{-1}}^B$ is the semiclassical magnetic Weyl quantization (see Section~\ref{s:quantization}). In particular, for every $a\in\mathscr{C}^\infty(\T^2,\C)$, one has
\begin{equation}\label{e:lastparttheo}
\lim_{\lambda\rightarrow +\infty}\int_{\T^2}a(x)|u_\lambda(x)|^2\dd x=\int_{\T^2}a(x) \dd x.
\end{equation}
\end{theorem}
Here $\mathscr{S}^0(T^*\T^2)$ denotes the subspace of $\mathscr{C}^\infty(\T^2\times\R^2,\C)$ made of functions all of whose derivatives are bounded. The precise definition of the magnetic Weyl quantization~\cite{MantoiuPurice2004, IftimieMantoiuPurice2007, HelfferPurice2010, IftimieMantoiuPurice2019} is recalled in Section~\ref{s:quantization}, together with its main properties. This is the analogue for the space $\mathscr{C}^\infty(\T^2,L)$ of the Weyl quantization on $\mathscr{C}^\infty(\T^2,\C)$~\cite{Zworski}. This theorem shows that magnetic eigenfunctions become equidistributed in the high frequency limit as soon as the geometric control condition $B_\infty>0$ is satisfied. It is for instance verified if $B>0$ everywhere on $\T^2$ but this geometric condition is in fact much more general as it only involves the average of the magnetic field along the geodesic flow:
$$
\varphi^t:(x,\theta)\in S^*\T^2\mapsto (x+t\theta,\theta)\in S^*\T^2.
$$

The reader familiar with semiclassical analysis will notice that the magnetic field, through the vector potential $A$, is only a subleading term in the high-energy limit $\lambda \to \infty$. In other words, $\mathscr{L}$ is a small perturbation of the Laplacian, and for this reason the underlying classical flow in this setting is the geodesic flow $\varphi^t$ on the torus.

In fact, as we shall recall it in Section~\ref{s:semiclassicalmeasure}, standard arguments of semiclassical analysis show that any accumulation point of the sequence 
$$
W_\lambda:a\in \mathscr{C}^\infty_c(T^*\T^2,\C)\mapsto
\langle\mathsf{Op}_{\lambda^{-1}}^B(a)u_\lambda,u_\lambda\rangle 
$$
is a probability measure on $S^*\T^2$ that is invariant by the geodesic flow, as in the case of the free Laplacian. These measures are the so-called semiclassical measures of $\mathscr{L}+V$. In fact, the geodesic flow has many invariant measures (e.g. invariant measures carried by closed orbit)
and Theorem~\ref{t:maintheo} shows that, among these invariant measures, only the Liouville measure is semiclassical. This is referred to as a \emph{quantum unique ergodicity} property for the operator $\mathscr{L}+V$. Such a property does not hold for the free Laplacian on the torus. It is remarkable that, despite being subleading, the magnetic field is still strong enough to entail quantum unique ergodicity. As far as we know, Theorem~\ref{t:maintheo} provides the first family of deterministic Laplace type operators on a compact manifold for which \emph{any} sequence of high frequency eigenfunctions equidistributes.

It is also important to note that the underlying classical flow (meaning the geodesic flow on $S^*\T^2$) is completely integrable and thus far from being ergodic with respect to the Liouville measure. Hence, among other things, Theorem~\ref{t:maintheo} also shows that quantum ergodicity does not imply classical ergodicity. Other examples of quantum ergodic systems with non-ergodic classical flow were described by Gutkin in~\cite{Gutkin2009} and this converse quantum ergodicity problem is discussed in more details by Zelditch in~\cite[\S4.5]{Zelditch2010}. In particular, it seems that, so far, all the known counterexamples to converse quantum ergodicity involved a $\Z_2$-symmetry which is not the case of Theorem~\ref{t:maintheo}.

As we shall see in the proof, Theorem \ref{t:maintheo} remains true if we only suppose that the eigenvalue equation~\eqref{e:eigenvalue} is satisfied up to a $o_{L^2}(1)$ remainder. In fact,~\eqref{e:lastparttheo} will be shown to be true if we only make the assumption that~\eqref{e:eigenvalue} is verified up to a $o(\lambda^{1/2})$ remainder.

\begin{remark}
Since the online publication of the present work, Le Balc'h, Niu and Sun considered the related question of observability for the magnetic Schr\"odinger equation when $\widehat{B}_0= 0$~\cite{LeBalchNiuSun2025}. In that case, one cannot have $B_\infty>0$ everywhere and they exhibited sequences of $o_{L^2}(1)$-quasimodes that concentrate along certain closed orbits in phase space~\cite[\S9]{LeBalchNiuSun2025} -- see also~\cite[\S5.3]{AnantharamanFermanianMacia2015} for related constructions in the electric case. It is most likely that their construction could be adapted to the more general case where $\widehat{B}_0\neq 0$. See also Appendix~\ref{s:general} for a discussion on the case where $B_\infty$ is not positive everywhere.
\end{remark}

Once Theorem~\ref{t:maintheo} is proven, it is also natural to figure out if one can compute the rate of convergence. Unfortunately, our proof is not of very quantitative nature. Yet, in the case where the magnetic field $B$ is constant and where $V$ is identically $0$, we are able to provide an alternative (and more direct) proof of Theorem~\ref{t:maintheo}. This will allow us to prove the following quantitative statement:
\begin{theorem}\label{t:maintheo-quantitative} Suppose that $B\equiv\widehat{B}_0>0$ and that $V\equiv 0$. Then, there exist constants $C_0,N_0>0$ such that, for any solution $u_\lambda$ to~\eqref{e:eigenvalue} and for any $a\in \mathscr{S}^0(T^*\T^2)$, one has
$$
\left|\langle \mathsf{Op}_{\lambda^{-1}}^B(a)u_\lambda,u_\lambda\rangle -\frac{1}{2\pi} \int_{S^*\T^2}a(x,\theta) \dd x\dd\theta\right|\leq C_0\left\|a\right\|_{\mathscr{C}^{N_0}}(1+|\lambda|)^{-\frac12},
$$
where $\operatorname{Op}_{\lambda^{-1}}^B$ is the semiclassical magnetic Weyl quantization.
\end{theorem}

Even if the conclusion of this theorem is much stronger than the one of Theorem~\ref{t:maintheo}, we emphasize that its proof is technically less involved as it only deals with constant magnetic fields. In fact, it is in some sense close to the original strategy for proving quantum ergodicity~\cite{Shnirelman1974, Shnirelman2022, Zelditch1987, ColindeVerdiere1985}. Indeed, for constant magnetic fields and for $V=0$, one can apply an exact Egorov Theorem for the magnetic Weyl quantization up to times of order $\lambda$. This allows us to average the symbol $a$ along the magnetic ($\lambda$-dependent) flow -- rather than the geodesic flow if we would have worked in bounded times. This flow, which is an artefact of the cyclotron motion but for small magnetic fields, slowly deviates from the geodesic flow. After that, we can use stationary phase arguments to prove equidistribution properties of the magnetic dynamics in the regime $\lambda\rightarrow \infty$. As we shall detail at the end of the introduction, proving Theorem~\ref{t:maintheo} requires a more subtle analysis as there is no good enough Egorov Theorem in that general set-up that would make a direct connection with a magnetic flow.

\subsection{Earlier results for vanishing magnetic fields}\label{ss:review-torus}

 When $B =V\equiv 0$, it follows from the works of Cooke~\cite{Cooke1971} and Zygmund~\cite{Zygmund1974} that any accumulation point $\nu$ of a sequence $\nu_\lambda(\dd x):=|u_\lambda(x)|^2\dd x$ is of the form $h(x)\dd x$ where $h\in L^2(\T^2,\R_+)$. This result was refined by Jakobson~\cite{Jakobson1997} who proved that $h$ is in fact a trigonometric polynomial. This cannot be improved as it can be witnessed by considering the sequence of eigenfunctions $u_{k}(x_1,x_2)=\sin(2\pi x_1)e^{2i\pi kx_2}$ whose limit measure is $\sin^2(2\pi x_1)\dd x_1\dd x_2.$ See also~\cite{Jaffard1990} for earlier results in dimension $2$. Jakobson also proved that, in any dimension, any accumulation point is of the form $h(x)\dd x$ with $h\in L^1(\T^d)$. When $V\neq  0$ but still $B\equiv 0$, Bourgain, Burq and Zworski obtained observability properties for the limit measure in dimension $2$ in the case of low regularity potentials~\cite{BurqZworski2012, BourgainBurqZworski2013}. Still for general (but more regular) potentials and using the notion of two-microlocal measures~\cite{Fermanian95, Miller96,Nier96, Fermanian00}, Anantharaman and Maci\`a proved that the limit measures are absolutely continuous and enjoy an observability property in any dimension~\cite{Macia2010, AnantharamanMacia2014}. Together with Fermanian-Kammerer~\cite{AnantharamanFermanianMacia2015} and L\'eautaud~\cite{AnantharamanLeautaudMacia2016}, they broadened the scope of applications of these two-microlocal methods and extended these results to eigenfunctions of non-degenerate completely integrable systems on $\T^d$ and of Schr\"odinger operators on the disk. Motivated by the Aharonov-Bohm effect, Maci\`a also proved that these results remain true in the case where the magnetic field is identically $0$ but the potential vector is not~\cite{Macia2014}.

Our proof of Theorem~\ref{t:maintheo} will in fact follow this kind of semiclassical strategy through second-microlocalization methods. More precisely, it is closely related to the analysis made by Maci\`a and the second author~\cite{MaciaRiviere2018} in order to analyze eigenmodes of semiclassical operators of the form $-h^2\Delta+\varepsilon_h^2V$ with $h\ll \varepsilon_h\ll 1$. Indeed, as we shall see in this article, looking at the high frequency limit of the operator $\mathscr{L}+V$ corresponds in some sense to a subprincipal perturbation of the operator $-h^2\Delta$ of size $h$. Following earlier results on Zoll manifolds~\cite{MaciaRiviere2016, MaciaRiviere2019}, the key point in~\cite{MaciaRiviere2018} was to understand these perturbative results and two-microlocal procedures through the scope of the quantum averaging method appearing in Weinstein's seminal work on Schr\"odinger operators on $\mathbf{S}^d$~\cite{Weinstein1977}. In some sense, this will be again the mechanism at work behind our proofs and the magnetic nature of the problem will yield the extra regularity properties compared with~\cite{MaciaRiviere2018}. 

Finally, we recall that, given an orthonormal basis of eigenfunctions of $\mathscr{L}+V$, Marklof and Rudnick proved that one can find a density one subsequence of eigenfunctions such that the limit measure is $\dd x$ in the configuration space\footnote{Strictly speaking, the proof in that reference is given for $B=V\equiv 0$ but the proof could be adapted to deal with the general case.} $\T^2$~\cite{MarklofRudnick2012}. Theorem~\ref{t:maintheo} shows that, as soon as $B$ verifies the geometric control condition $B_\infty>0$, there is no need to extract a subsequence and equidistribution holds in the full phase space $S^*\T^2$.

\subsection{Quantum unique ergodicity}

 Marklof--Rudnick's Theorem is in fact the toral analogue of the classical quantum ergodicity Theorem of Shnirelman \cite{Shnirelman1974, Shnirelman2022}, Zelditch~\cite{Zelditch1987} and Colin de Verdi\`ere~\cite{ColindeVerdiere1985}. This theorem states that, as soon as the Liouville measure is ergodic for the geodesic flow of a compact Riemannian manifold $(M,g)$, then \emph{most of} the eigenfunctions of a \emph{given} orthonormal basis of Laplace eigenfunctions are equidistributed in phase space $S^*M$. This result is in fact true for more general elliptic operators whose underlying Hamiltonian flow is ergodic. See e.g. the works of Helffer, Martinez and Robert~\cite{HelfferMartinezRobert1987} for semiclassical Schr\"odinger operators. A few years later~\cite{RudnickSarnak1994}, Rudnick and Sarnak made the so-called \emph{quantum unique ergodicity conjecture} by predicting that, for \emph{compact negatively curved manifolds}, \emph{all} Laplace eigenfunctions should become equidistributed in the high frequency limit. Equivalently the Liouville measure should be the only semiclassical measure for the Laplace-Beltrami operator. Among other things, this was motivated by their result showing that, on arithmetic surfaces, there exists a specific basis of eigenfunctions, the Hecke eigenbasis, for which eigenfunctions put zero mass on closed hyperbolic geodesics in the high-frequency limit. Later on, building on earlier works with Bourgain~\cite{BourgainLindenstrauss2003}, Lindenstrauss proved that these Hecke eigenmodes become equidistributed in the high frequency limit~\cite{Lindenstrauss2006}. This property is often referred as \emph{arithmetic quantum unique ergodicity}.

For more general negatively curved manifolds or more general eigenbases, the conjecture remains completely open as far as we know. Yet, several important progress were made towards the understanding of this problem in the last twenty years. First, in the related toy-model of quantum cat-maps, it was proved by Faure, Nonnenmacher and De Bi\`evre that the quantum unique ergodicity conjecture does not hold true in dimension $2$~\cite{FaureNonnenmacherDeBievre2003}. On the higher dimensional analogues of this model, Kelmer even proved that arithmetic quantum unique ergodicity fails as soon as there exist invariant isotropic rational subspaces~\cite{Kelmer2010}. In the case of Laplace eigenfunctions on the Bunimovich stadium (for which ergodicity holds despite the weak hyperbolicity features of the billiard flow), Hassell proved that quantum unique ergodicity also fails~\cite{Hassell2010}. On the positive direction, Anantharaman proved that, on negatively curved manifolds, any accumulation point in phase space must have positive Kolmogorov-Sinai entropy and thus discarded full concentration along closed geodesics~\cite{Anantharaman2008}. Then, together with Koch and Nonnenmacher, they obtained more explicit lower bounds on this entropy~\cite{AnantharamanNonnenmacher2007, AnantharamanKochNonnenmacher2009}. See also~\cite{Riviere2010a, Riviere2010b} for more precise results in dimension $2$. More recently, for negatively curved surfaces, Dyatlov, Jin and Nonnenmacher proved that, in the high frequency limit, Laplace eigenfunctions should put some positive mass on every nontrivial open subset of $S^*M$~\cite{DyatlovJin2018, DyatlovJinNonnenmacher2022}. See also~\cite{DyatlovJezequel2024, AthreyaDyatlovMiller2024,Kim2024} for higher dimensional results in the same vein. These recent results are again specific to chaotic (Anosov) classical dynamics and they are based on a subtle harmonic analysis property originally developed by Dyatlov and Zahl~\cite{DyatlovZahl2016} to study spectral gaps in chaotic scattering and referred as a fractal uncertainty principle.

\medskip
It is worth mentionning that the dynamical mechanism at work in our setting is of very different nature as in these works on quantum ergodicity. Indeed, the underlying Hamiltonian flow for the operator $\mathscr{L}+V$ is the geodesic flow on the unit tangent bundle of the torus which is a \emph{completely integrable} dynamical system. Thus, the dynamics is of very different nature: existence of Kronecker tori, but also tori of periodic orbits. Moreover, in the case of the free Laplacian on $\T^2$, the discussion in \S\ref{ss:review-torus} shows that Laplace eigenfunctions are far from being quantum uniquely ergodic. For these reasons, the dynamical properties of the underlying classical flow cannot be the only reason for the equidistribution properties of the eigenfunctions of $\mathscr{L}+V$ and this is rather the action of the magnetic field that will produce equidistribution result in our setting. 

In this direction, we emphasize that the geometric control condition $B_\infty>0$ was already identified by Glass and Han-Kwan in~\cite[Prop.~5.1]{GlassHanKwan2012} as playing an important role when studying the classical dynamics induced by a magnetic field. Indeed, under this geometric assumption and given any nonempty open set $\omega\subset\T^2$ and any $T_0>0$, they proved that one can find $R_0>0$ such that, for every $(x_0,\xi_0)\in T^*\T^2$ verifying $|\xi_0|\geq R_0$, one has 
\begin{equation}
\label{e:daniel}
x([0,T_0])\cap \omega \neq \emptyset,
\end{equation}
where $(x(t),\xi(t))$ solves the Cauchy problem
\begin{equation}
\begin{cases}
(x',\xi')=(\xi,B(x)\xi^\perp),\\ (x(0),\xi(0))=(x_0,\xi_0).
\end{cases}
\end{equation}

In other words, a particle with a high momentum inside a magnetic field verifying $B_\infty>0$ explores all the configuration space. Even if we will not exploit this dynamical property explicitly as our limit measures are not invariant by the dynamics~\eqref{e:daniel}, it already illustrates how the geometric control condition on $B$ induces a kind of minimality of the magnetic dynamics on the classical side.

Regarding the influence of magnetic fields towards quantum unique ergodicity, we can mention the works of Zelditch for magnetic Laplacians on surfaces of constant negative curvature~\cite{Zelditch1992b}. In that reference, he proved that, if $B$ is constant but strong (meaning $\widehat{B}_0$ is equivalent to $\lambda/\sqrt{2}$), then the corresponding eigenfunctions become equidistributed in phase space (with no need to extract a subsequence). In that set-up, the reason for equidistribution comes from the ergodic properties of the classical flow. Indeed, in his case, the underlying classical flow is the horocycle flow (and not the geodesic flow) which has a \emph{unique} invariant measure and semiclassical measures are invariant by this flow. See also~\cite{CharlesLefeuvre2024} for recent developments by Charles and Lefeuvre in this direction and~\cite{MarklofRudnick2000, Rosenzweig2006} for related results on quantum maps by Marklof--Rudnick and Rosenzweig using unique ergodicity of the classical map. Besides the fact that the classical flow is not uniquely ergodic for the Liouville measure, a notable difference of our work with~\cite{Zelditch1992b, CharlesLefeuvre2024} is that we work in the regime where $\widehat{B}_0$ is fixed while these references consider the strong magnetic field regime where $|\widehat{B}_0|\rightarrow\infty$. As highlighted by Ceki\'c and Lefeuvre in~\cite{CekicLefeuvre2024}, an interesting issue would be to understand more generally the asymptotic regularity of eigenfunctions as $\|(\lambda,\widehat{B}_0)\|\rightarrow \infty$ which would encompass all the different regimes. 

Finally, we can mention that another efficient mechanism for producing quantum uniquely ergodic bases of eigenfunctions is given by probabilistic tools. For instance, in the case of the round sphere, one has large multiplicity in the Laplace spectrum and there is thus a large choice of orthonormal basis. This set of eigenbasis can be endowed with a natural probability measure and it was proved by Zelditch that most orthonormal bases have the quantum ergodicity property~\cite{Zelditch1992a}. This was later refined by VanderKam to show that \emph{almost all orthonormal bases} on $\mathbf{S}^d$ enjoy the quantum unique ergodicity property~\cite{VanderKam1997}. In a different direction but still of probabilistic nature, we refer to the works of Bourgade, Yau and Yin~\cite{BourgadeYau2017, BourgadeYauYin2020} and the references therein for quantum unique ergodicity results for eigenfunctions of Wigner matrices, i.e.
random matrices with matrix elements distributed by identical mean-zero random variables.

\subsection{Strategy of proof}

We start in Section \ref{s:laplacian} by giving a precise definition of the magnetic Laplacians $\mathscr{L}_\alpha$ on the torus. In particular these operators act on a space of functions with a twisted periodicity condition denoted by $L^2(\T^2,L)$, see \eqref{eq.CL}. We also give an explicit description of the spectrum in the case of constant magnetic fields.

Then, in Section \ref{s:quantization} we introduce the magnetic Weyl quantization, in the spirit of \cite{MantoiuPurice2004, IftimieMantoiuPurice2007, HelfferPurice2010, IftimieMantoiuPurice2019}. Given a symbol $a \in \mathscr C^{\infty}(\T^2 \times \R^2)$, we show that the quantization $\Op (a)$ defines a nice operator on smooth subspaces of $L^2(\T^2,L)$, under some standard regularity assumptions on $a$. We also prove a product formula for this quantization, and a Calder\'on-Vaillancourt Theorem, following the standard techniques of pseudodifferential calculus.

These tools being introduced, we proceed in the remaining sections to the study of semiclassical measures. We begin in Section \ref{s:proof-quantitative} with a proof of Theorem~\ref{t:maintheo-quantitative} which deals with the case of \emph{constant} magnetic fields. Under this assumption, we can directly use the properties of this quantization procedure together with the so-called quantum averaging method and classical estimates on oscillatory integrals on the circle. More precisely, in that case, one has an exact Egorov Theorem involving a $\lambda$-dependent flow\footnote{For a fixed time, this flow converges to the geodesic flow as $\lambda\rightarrow+\infty.$} which is periodic of period of order $\lambda$. When averaging the magnetic Wigner distribution, one can replace the symbol $a$ by the average of $a$ over this long period. Then, a stationary phase argument shows that this average converges uniformly to $\int_{S^*\T^2}a\dd x\dd\theta$.

\medskip
In the general case, no good enough Egorov Theorem is avalaible, nor a periodic flow. Hence, one needs to proceed in a different way. Up to dividing by $\lambda^2= h^{-2}$, we can write the eigenvalue equation \eqref{e:eigenvalue} in a semiclassical way,
\[h^2 (\mathscr{L} +V )u_h = u_h,\]
and we consider a sequence $(u_h)_{h\to 0}$ satisfying this equation -- or more generally an approximate version, i.e.
\[h^2 (\mathscr{L} +V )u_h = u_h+r_h,\quad\|r_h\|_{L^2}=o(h^2),\]
We will first show that the sequence of measures $|u_h(x)|^2 \dd x$ has only $\dd x$ as an accumulation point. To do so, we introduce in Section \ref{s:semiclassicalmeasure} the phase-space Wigner distribution,
\[ \langle w_h, a \rangle = \langle \Op (a) u_h, u_h \rangle, \quad a \in \mathscr{C}^{\infty}_c(\T^2 \times \R^2),\]
and study its accumulation points $\dd\mu$ in the limit $h \to 0$. In Lemma \ref{l:magnetic-wigner} we recall that a limit distribution must be a probability measure $\mu$ on $\T^2 \times \mathbf S^1$, which is invariant under the geodesic flow. This standard result follows from the fact that the principal symbol of $h^2( \mathscr{L} + V)$ is $|\xi|^2$. Indeed, we have the commutator estimate
\begin{equation}\label{e:proof-invariance}
\frac{2h}{i}  \langle w_h, \xi\cdot\partial_x a \rangle=\left\langle \left[h^2(\mathscr{L}+V),\Op (a)\right] u_h, u_h \right\rangle=\mathscr{O}(\|r_h\|_{L^2}).
\end{equation}
Hence, as soon as $\|r_h\|_{L^2}=o(h)$, we can conclude that the limit $\mu$ satisfies $\xi\cdot\partial_x \mu=0$ which amounts to prove the invariance by the geodesic flow.

\medskip

In fact, any geodesic-invariant measure on the torus can be decomposed according to periodic and dense orbits. For irrational angles $\theta$, the geodesic $t\in\R \mapsto x+t\theta\in\T^2$ is dense and the measure $\mu$ restricted to these angles is the Lebesgue measure along the $x$-variable (see Lemma \ref{l:anantharamanmacia}). We aim at proving that this is also the Lebesgue measure in $x$ when restricted to rational angles -- where the geodesic flow is periodic. To do this, we only have to describe the measure $\mu$ along periodic orbits and this is the aim of Section \ref{s:periodic}. This section is the most delicate part of the proof and our strategy consists in using the subprincipal terms of the operator to unveil extra invariance properties besides the one of the geodesic flow.

\medskip
In view of understanding the mechanism at work, suppose for a moment that $B\equiv \widehat{B}_0$. In that case, the properties of the magnetic Weyl quantization give the following second order term in the commutator estimate,
\begin{equation}\label{e:key-relation}
\left[h^2\mathscr{L},\Op (a)\right]\simeq\frac{2h}{i}\Op\left(\xi\cdot\partial_x a +h\widehat{B}_0(\xi_2\partial_{\xi_1} a-\xi_1\partial_{\xi_2} a)\right) +\mathscr{O}(h^3).
\end{equation}
This is precisely the term $h\widehat{B}_0(\xi_2\partial_{\xi_1} a-\xi_1\partial_{\xi_2} a)$ that will be responsible with the extra regularity property we are aiming at, both along the $x$ and the $\xi$-variables. In order to understand the restriction of the measure to the rational angle $\theta=(0,1)$ (for instance), we will in fact rescale the variables along the $\xi_1 = \xi \cdot \theta^\perp$ direction and it will be sufficient to consider test functions of the form $a_h(x,\xi)=\tilde{a}(x_1,\xi,\xi_1/h^{\frac12})$ in view of ruling out concentration phenomena along such a rational angle. For such an observable, the above equality rewrites, up to remainder terms,
$$
\left[h^2\mathscr{L},\Op (a_h)\right]\simeq\frac{2h^{\frac{3}{2}}}{i}\Op\left(\eta\partial_{x_1} \tilde{a} +\widehat{B}_0\partial_{\eta} \tilde{a}\right),\quad\eta=\frac{\xi_1}{h^{\frac12}}.
$$
Hence, if we implement this property in the relation~\eqref{e:proof-invariance}, we find that, as soon as $r_h=o_{L^2}(h^{3/2})$, the limit measure verifies the extra invariance property $\eta\partial_{x_1} \mu+\widehat{B}_0\partial_\eta \mu=0$ along the direction $\theta=(0,1)$. It allows to derive some extra regularity along the $x_1$-variable in this direction using that $\widehat{B}_0\neq 0$. For non-constant magnetic fields, we will end up with a relation of the form 
\begin{equation}\label{eq.exB}
\eta\partial_{x_1} \mu+B_\infty(x,\theta)\partial_\eta \mu=0.
\end{equation}

Such a strategy was initiated by Maci\`a in~\cite{Macia2010} through the use of $2$-microlocal measures. See also~\cite{Wunsch2008, VasyWunsch2009} for related results in terms of Lagrangian regularity. More precisely, we can make the above formal argument rigorous by  proceeding as in~\cite{MaciaRiviere2018}. Namely, through the introduction of these rescaled extra variables $\eta$, we split the eigenfunctions between the part that is localized close to a given rational angle (at distance $\lesssim h^{\frac 12}$), and away from it (at distance $\gg h^{\frac 12}$). 

In Section \ref{sec.noncompactpart}, we analyse the part at distance $\gg h^{\frac 12}$ (i.e. when the rescaled variable $\eta$ is larger than $R$, for some very large constant $R$). This part gives rise to the Lebesgue measure $\dd x$ when $h\rightarrow 0^+$ and $R \to \infty$. Here, the magnetic field has no impact: it only has to be small enough.

Then, in Section \ref{sec.compactpart} we analyse the part at distance $\leq h^{\frac 12}$ (when $\eta \leq R$). In Lemma~\ref{l:invariance-2microlocal}, we show that the induced microlocal measure is invariant by another flow involving the magnetic field $B$, similar to \eqref{eq.exB}. When $B_\infty>0$, this extra invariance property results into the fact that the $2$-microlocal measure is $0$. Indeed, the transport equation \eqref{eq.exB} forces mass to leave any bounded region in $\eta$.

This will conclude the proof of the second part of Theorem~\ref{t:maintheo}, meaning that any semiclassical measure $\mu$ projects to the Lebesgue measure along the $x$-variable as soon as $r_h=o(h^{\frac32})$:
\begin{equation} \label{eq.mufactor}
 \dd \mu (x,\xi)  = \dd x \otimes  \dd \nu(\xi), 
 \end{equation}
for some probability measure $\nu$ on $\mathbf S^1$. Finally in Section~\ref{s:proof}, we use the properties of the magnetic Weyl quantization to prove that $\nu$ must be the uniform measure. More precisely, once we know that $\mu$ is of the form \eqref{eq.mufactor}, we can restrict ourselves to test functions that are of the form $a(\xi)$ in~\eqref{e:proof-invariance}. Then, as soon as $r_h=o_{L^2}(h^2)$,~\eqref{e:key-relation} combined with~\eqref{e:proof-invariance} will ensure the rotational invariance 
\[0=\xi_2\partial_{\xi_1}\nu-\xi_1\partial_{\xi_2}\nu=-\partial_\theta\nu,\]
from which we infer that $\nu$ is the Lebesgue measure on $\mathbf{S}^1$.

\medskip
In Appendix~\ref{s:general}, we discuss what happens when $B_\infty(x,\theta)$ is not positive everywhere and what partial conclusions can be obtained in that general case using the methods of the present article.

\begin{remark}
 All along the article, we suppose that $\widehat{B}_0>0$ for the sake exposition. Yet, the main results remain true when $\widehat{B}_0<0$ with the assumption that $B_\infty<0$ everywhere.
 \end{remark}

\subsection*{Acknowledgements}
 We are most grateful to Bernard Helffer for bringing to our attention the magnetic Weyl calculus originally developed in~\cite{MantoiuPurice2004} as a convenient and elegant way of quantizing observables in the presence of magnetic fields. The second author also thanks Fabricio Maci\`a for discussions over the years regarding the study of semiclassical measures for completely integrable systems. We benefited from two very detailed reports from anonymous referees that greatly helped us to improve the exposition of the article. We warmly thank these two referees for their suggestions and questions.

LM was supported by the ERC Advanced Grant MathBEC - 101095820. (This work is funded by the European Union. Views and opinions expressed are however those of the author only and do not necessarly reflect those of the European Union or the European Research Council. Neither the European Union nor the granting authority can be held responsible for them.)

GR acknowledges the
support of the Institut Universitaire de France, the PRC grant ADYCT (ANR-20-CE40-0017) and the Centre Henri Lebesgue (ANR-11-LABX-0020-01).

\section{Magnetic Laplacian on the torus}\label{s:laplacian}

 A magnetic field on a configuration space consists in a $2$-form $\mathbf{B}$. The difficulty when defining a Schr\"odinger operator quantizing this magnetic field comes from the fact that $\mathbf{B}$ may not derive from a potential (i.e. it cannot be written $\mathbf{B}=\dd A$). Note also that, even if there exists such a primitive, there is not necessarly a unique choice for $A$. The first difficulty is handled by dealing with functions taking their values in a certain line bundle and by considering an appropriate Laplace type operator acting on this line bundle. In the case where the configuration space is the $2$-torus, this construction is particularly simple and explicit and we will review it in this section. This will give rise to a magnetic Schr\"odinger operator which can be viewed as a periodized version of the magnetic Laplacian on $\R^2$. Finally, as we shall see, there is not a single choice of such operators and there is a whole family of allowed line bundles parametrized by some parameter $\alpha\in\R^2/2\pi\Z^2$. This is related to the Aharonov--Bohm effect in the physics literature.

\subsection{Definitions}

 In the $2$-dimensional case, a magnetic field $\mathbf{B}$ can be identified with a function $B$ once we have fixed the volume form $\dd x_1\wedge\dd x_2$. Thus, we let $B \in \mathscr{C}^\infty(\T^2,\R)$ be a periodic magnetic field that is fixed once and for all. It can be written as the sum of its Fourier series, 
\begin{equation}
B(x) = \sum_{p \in \Z^2} \widehat{B}_p e^{2i \pi p \cdot x }, \qquad \widehat{B}_p = \int_{\T^2} B(x) e^{-2i\pi p \cdot x} \dd x. 
\end{equation}
One has
\begin{lemma}\label{l:potential} Let $B \in \mathscr{C}^\infty(\T^2,\R)$. Then, there exists $A^{\rm{per}}\in\mathscr{C}^\infty(\T^2,\R^2)$ such that
   $$B = \partial_1 A_2 - \partial_2 A_1,$$ 
for  $A = A^0 + A^{\rm{per}}$,
 and $A^0(x_1,x_2) = \frac{\widehat{B}_0}{2}\big( -x_2 , x_1 \big)$.
\end{lemma}
\begin{proof}
One has that $\int_{\T^2}(B-\widehat{B}_0)\dd x_1\wedge \dd x_2=0$. Hence, there exists a one-form $\alpha=A_1^{\rm{per}}\dd x_1+ A_2^{\rm{per}}\dd x_2$ in $\Omega^1(\T^2)$ such that $\dd\alpha= (B-\widehat{B}_0)\dd x_1\wedge \dd x_2.$ Hence, we can set $A=(A_1^{\rm{per}},A_2^{\rm{per}})$ which has the expected property. 
\end{proof}

\begin{remark}
Note that the choice of $A^{\operatorname{per}}$ is not unique and it can be modified by a gradient vector field $\nabla\phi$, with $\phi\in\mathscr{C}^\infty(\T^2,\R)$ (the so-called gauge choice). Yet, as usual when studying magnetic Laplacians, it would result into unitarily equivalent operators. More precisely, if we change $A^{\rm{per}}$ to $A^{\rm{per}}+\nabla\phi$, the resulting magnetic Laplacian will be of the form $e^{i\phi}\mathscr{L}e^{-i\phi}$ (see the precise expression for $\mathscr{L}$ below).
Hence, eigenfunctions are modified by multiplication by $e^{i\phi}$ which does not affect the set of semiclassical measures thanks to the composition rule for pseudodifferential operators.
\end{remark}

We define the magnetic translations associated to $\widehat{B}_0$ by
\begin{equation}
\mathcal T_m^B u (x) = e^{ \frac{i \widehat{B}_0}{2} m \wedge x} u(x-m), \qquad  u \in L^2(\R^2), \qquad m \in \Z^2,
\end{equation}
where $m \wedge x = m_1 x_2 - m_2 x_1$.
The class of smooth functions invariant under these translations is denoted by
\begin{equation}\label{eq.CL}
\mathscr{C}^\infty(\T^2,L) = \big\lbrace u \in \mathscr{C}^\infty(\R^2), \quad \mathcal T_m^B u = u \quad \forall m \in \Z^2 \big\rbrace.
\end{equation}
Note that, for this space to make sense, we need the commutation property $\mathcal T_{m}^B \mathcal T^B_{m'} = \mathcal T_{m'}^B \mathcal T^B_{m}$, which translates in the flux condition $\widehat{B}_0 \in 2\pi \mathbf Z$. From a geometrical point of view, these functions can be seen as sections of a line bundle over $\T^2$. The notation $L$ refers to this line bundle, even though we do not need this interpretation in the article. In fact, when considering the limit of large magnetic fields, one usually writes the line bundle as $L^{\otimes \phi}$ (where $\widehat{B}_0=2\pi\phi$) to emphasize the dependence in $\phi$. In the present work, $\phi$ is fixed and we drop this exponent to alleviate the notations. It is worth noting that elements in $\mathscr{C}^\infty(\T^2,L)$ are stable by multiplication by a function in $\mathscr{C}^\infty(\T^2,\C)$.

\begin{remark} The line bundle on $\T^2$ is the set $L=\R^2\times \C/_\sim$ where
$$
\forall m\in\Z^2,\quad(x,u)\sim \left(x-m, e^{\frac{i\widehat{B}_0}{2}m\wedge x}u\right).
$$
The corresponding projection is given by $(x,u)\in L\mapsto x\in\T^2$ and the set $\mathscr{C}^\infty(\T^2,L)$ can be identified with the set of smooth sections $\T^2\rightarrow L$. 
\end{remark}

Similarly we define the space of square integrable functions (resp. $H^2$ functions) invariant by the magnetic translations,
\begin{equation}
L^2(\T^2,L) = \big\lbrace u \in L^2_{\rm{loc}}(\R^2), \quad \mathcal T_m^B u = u \quad \forall m \in \Z^2 \big\rbrace,
\end{equation}
\begin{equation}
H^2(\T^2,L) = \big\lbrace u \in H^2_{\rm{loc}}(\R^2), \quad \mathcal T_m^B u = u \quad \forall m \in \Z^2 \big\rbrace.
\end{equation}
They are the closure of $\mathscr{C}^\infty(\T^2,L)$ for the norms
\begin{equation}
\| u \|_{L^2} = \Big( \int_{\T^2} | u(x) |^2 \dd x \Big) ^{\frac 1 2}, \qquad \| u \|_{H^2} = \Big( \sum_{|\alpha| \leq 2 } \int_{\T^2} | (\partial^{A})^\alpha u(x) |^2 \dd x \Big) ^{\frac 1 2},
\end{equation}
where $\partial^A=\left(\partial_{x_1}-iA^0_1, \partial_{x_2}-iA^0_2\right)$ and where we identify $\T^2$ with the fundamental domain $[0,1]^2$. Note that, as $\partial^A$ commutes with the magnetic translations, these norms and the associated scalar products are independent of the choice of fundamental domain for $\T^2$.

We define a familly of operators, the magnetic Laplacians on the torus.

\begin{theorem}
Assume $\widehat{B}_0 \in 2\pi \Z$ and let $A$ be as in Lemma~\ref{l:potential}. Then, for all $\alpha \in \R^2$, the operator given by
\[\mathscr{L}_\alpha u = \big( -i \partial_{x_1} - \alpha_1 - A_1 \big)^2 u + \big(-i\partial_{x_2} - \alpha_2 - A_2 \big)^2u, \qquad u \in H^2_{\rm{loc}}(\R^2),\]
commutes with the magnetic translations, i.e. $\mathscr{L}_\alpha \mathcal T_m^B = \mathcal T_m^B \mathscr{L}_\alpha$ for $m \in \Z^2$.
Moreover, $\mathscr{L}_\alpha$ is self-adjoint on $L^2(\T^2,L)$ with domain
\[D(\mathscr{L}_\alpha) = H^2(\T^2,L),\]
and has compact resolvent. Finally, if $\alpha - \alpha' \in 2\pi \Z^2$, then $\mathscr{L}_\alpha$ and $\mathscr{L}_{\alpha'}$ are unitarily equivalent.
\end{theorem}

\begin{proof}
The operator $\mathscr{L}_\alpha$ is well defined on $H^2_{\rm{loc}}(\R^2)$. To check the commutation property with the magnetic translations, it is enough to prove that $(-i\partial_{x_j}-A^0_j)$ commutes with $\mathcal{T}_m^B$ (as $\alpha_j+A^{{\rm per}}_j$ is periodic),
\begin{multline*}
\left(-i\partial_{x_1}+\frac{\widehat{B}_0 x_2}{2}\right)\mathcal{T}_m^B(u)(x)
\\=e^{\frac{i\widehat{B}_0}{2}m\wedge x}\left(-i\partial_{x_1}u(x-m)+\frac{\widehat{B}_0}{2}(x_2-m_2)u(x-m)\right)\\
=
\mathcal{T}_m^B\left(-i\partial_{x_1}+\frac{\widehat{B}_0 x_2}{2}\right)(u)(x),
\end{multline*}
and the same holds for $j=2$. In particular, $\mathscr{L}_\alpha$ maps $H^2(\T^2,L)$ on $L^2(\T^2,L)$, and it is thus a well defined unbounded operator on this Hilbert space. By convolution arguments, $\mathscr{C}^\infty(\T^2,L)$ (and thus $H^2(\T^2,L)$) is a dense subspace of $L^2(\T^2,L)$. Moreover, by the Rellich-Kondrachov Theorem (applied on open bounded subsets of $\R^2$), the injection $H^2(\T^2,L)\subset L^2(\T^2,L)$ is compact. It follows that the resolvent $(\mathscr{L}_{\alpha}+i)^{-1}$ is a compact operator. Using integration by parts, we know that $\mathscr{L}_\alpha$ is symmetric, since the boundary terms vanish by the translation property $\mathcal T_m^B u = u$. The adjoint has domain
\[D(\mathscr{L}_\alpha^*) = \lbrace u \in L^2(\T^2,L), \quad \mathscr{L}_\alpha u \in L^2_{\rm{loc}}(\R^2) \rbrace ,\]
and, by elliptic regularity, this is precisely $H^2(\T^2,L)$. Thus $\mathscr{L}_\alpha$ is self-adjoint. To prove the last statement, note that, if $\alpha - \alpha' \in 2\pi \Z^2$, the function $e^{i (\alpha - \alpha') \cdot x}$ is periodic, and we have the conjugation property
\[e^{-i (\alpha - \alpha')  x} \mathscr{L}_{\alpha} e^{i (\alpha - \alpha')  x} = \mathscr{L}_{\alpha'}.\]
\end{proof}

\subsection{The case of a constant magnetic field}

As an example of the operators at hand, we describe in this short paragraph the spectrum of $\mathscr{L}_\alpha$ in the case of a constant field $B = \widehat{B}_0$ (with the gauge $A^{\rm{per}} = 0$). We suppose that $B\neq 0$ as, in the case $B=0$, eigenmodes are given by the standard decomposition in Fourier series. In fact when $B\neq 0$, we can also find explicit formulas for the eigenfunctions and eigenvalues. We emphasize that these formulas will \emph{not} be used in the remaining of the article, as we do not need them to prove Theorem~\ref{t:maintheo-quantitative}. This paragraph is only here for the sake of concreteness.

Firstly, we can conjugate $\mathscr{L}_\alpha$ by a phase to change the gauge,
\[ \mathscr{P}_{\alpha,B} = e^{i ( \frac{B}{2} x_2 - \alpha_1 ) x_1} \mathscr{L}_\alpha e^{-i (\frac{B}{2} x_2 - \alpha_1 ) x_1} = - \partial_{x_1}^2 + (-i \partial_{x_2} - \alpha_2 - Bx_1)^2. \]
Conjugating by such a phase results into a different periodicity condition, meaning that $L^2(\T^2,L)$ has to be modified using this phase to make sense of $\mathscr{P}_{\alpha,B}$ on $\T^2$. In fact, the domain $D( \mathscr{P}_{\alpha,B})$ of $\mathscr{P}_{\alpha,B}$ is now
$$\big\lbrace v \in H^2_{\rm{loc}}(\R^2), \quad v(x) = e^{-i \frac{B}{2} m_1m_2} e^{-i \alpha_1 m_1} e^{iBm_1 x_2} v(x-m), \quad m \in \mathbf Z^2 \big\rbrace,$$
and $\mathscr{P}_{\alpha,B}$ is unitarily equivalent to $\mathscr{L}_\alpha$.
In particular, if $v$ is in the domain, then $x_2 \mapsto v(x_1,x_2)$ is $1$-periodic and can be written as the sum of its Fourier series,
\[ v(x) = \sum_{p \in \mathbf Z} e^{2i\pi px_2} v_p(x_1).\]
The periodicity condition in $x_1$ translates in
\[ v_{p+ \phi} (x_1) = e^{-i\alpha_1} v_p(x_1-1), \quad p\in \mathbf Z, \quad x_1 \in \mathbf R, \]
where $\phi = \frac{B}{2\pi} \in \mathbf Z$. Therefore, any function in the domain of $\mathscr{P}_{\alpha,B}$ is determined by the functions $v_0(x_1), \cdots, v_{\phi-1}(x_1)$, through the formula
\[v(x) = \sum_{n\in \mathbf Z} \sum_{p=0}^{\phi-1} e^{2i\pi ( n \phi + p) x_2} e^{-in \alpha_1} v_p(x_1-n).\]
Now assume $v$ is an eigenfunction with eigenvalue $\lambda$. The eigenvalue equation on the $p$-mode reads
\[ -\partial_{x_1}^2 v_p + (2\pi p - \alpha_1 - B x_1)^2 v_p = \lambda v_p ,\]
which is a translated harmonic oscillator. The eigenvalues are thus $\lambda_j = (2j-1) B$ for $j \geq 1$, with multiplicity $\phi$, and the associated eigenfunctions are given by twisted sums of translated Hermite functions. 

Note that the explicit expression of the eigenfunctions could maybe be used to prove Theorem~\ref{t:maintheo-quantitative} but we did not succeed in doing so. Hence, we will follow a different approach based on the properties of our quantization -- see Section~\ref{s:proof-quantitative}.

\section{A magnetic quantization}\label{s:quantization}

In this section, we introduce a magnetic quantization which is adapted to our problem. The goal is to make sense of pseudo-differential operators acting on spaces of functions with the magnetic periodicity condition, namely $\mathscr{C}^{\infty}(\mathbf T^2,L)$ introduced in \eqref{eq.CL}. Using the complex line bundle point of view, this could be achieved by considering pseudo-differential operators with operator-valued symbols, i.e. functions lying in $\mathscr{C}^{\infty}(\T^2\times\R^2,\operatorname{Hom}(L,L))$. Here, due to the Euclidean structure of our problem, we can proceed in a slightly more explicit manner by considering twisted versions of the standard Weyl quantization~\cite{Zworski}. Such a construction was carried out in a microlocal set-up on $\R^d$ by Mantoiu and Purice~\cite{MantoiuPurice2004}. See also~\cite{BoutetdeMonvelGrigisHelffer1976} for earlier related constructions,~\cite{IftimieMantoiuPurice2007, IftimieMantoiuPurice2019} for further developments regarding magnetic fields and~\cite{HelfferPurice2010} for the semiclassical version. One advantage of this point of view is to ensure that all terms in the asymptotic expansions of standard theorems of pseudo-differential calculus depend only on $B$ (and not on a choice of a local primitive of $B$ if one would have chosen to work in local charts). Here, we will need a semiclassical and \emph{periodic} version of this magnetic pseudo-differential calculus. Hence, we review how the construction in these works can be adapted to our periodic set-up and what are the main properties of these magnetic pseudo-differential operators.

\begin{remark}
Taking $\widehat{B}_0=0$ in this section, the construction below recovers the standard Weyl quantization on the torus.  
 \end{remark}

 \begin{remark} Taking $\widehat{B}_0\in 2\pi\Z$ is a standard assumption when dealing with magnetic Laplacians on compact manifolds -- see e.g.~\cite{Demailly1985}. Yet, it would be interesting to make sense of the questions raised in this article when the magnetic flux $\widehat{B}_0$ is not in $2\pi\Z$. As pointed to us by one of the referee, there is a formal resemblance between the construction below and the noncommutative pseudodifferential calculus described in~\cite{HaLeePonge2019,HaLeePonge2019b} and it could be a way to attack this question.   
  \end{remark}

\subsection{Magnetic Weyl quantization}

The symbols $a(x,\xi)$ of such pseudo-differential operators are periodic with respect to $x$, i.e. $a \in \mathscr{C}^{\infty}(\mathbf T^2 \times \mathbf R^2,\C)$. The construction relies on the magnetic Weyl operators $W_h^{B}(\eta, \zeta)$, defined for $\eta \in 2 \pi \Z^2$ and $\zeta \in \R^2$ by
\begin{equation}
W_h^B(\eta,\zeta) u(x) =e^{  \frac{ih \widehat{B}_0}{2} \zeta \wedge x}  e^{i  \eta \cdot( x+ \frac{h \zeta}{2})}u(x+h\zeta), \qquad u \in L^{2}_{\text{loc}}(\mathbf R^2).
\end{equation}
It should be noted that $W_h^B(\eta,\zeta)$ commutes with the magnetic translations,
\[W_h^B(\eta,\zeta) \circ  \mathcal T_m^B = \mathcal T_m^B \circ W_h^B(\eta,\zeta),\]
which makes it a well defined isometry $W_h^B(\eta,\zeta) : L^{2}(\mathbf T^2,L) \rightarrow L^{2}(\mathbf T^2,L)$. The adjoint is also explicitly given by
\begin{equation}\label{eq.Whstar}
W_h^B(\eta,\zeta)^* = W_h^B(-\eta, - \zeta).
\end{equation}
Finally we have the following composition formula for Weyl operators.

\begin{proposition}\label{prop.compT}
For all $\eta$, $\eta' \in 2 \pi \Z^2$, $\zeta$, $\zeta' \in \R^2$, we have the composition rule
\[ W_h^B(\eta,\zeta) \circ W_h^B(\eta',\zeta') =
e^{i\frac h 2 \Omega_{hB}\left(\eta,\zeta;\eta',\zeta'\right)}W_h^B(\eta +\eta', \zeta+\zeta'),\]
where 
$$
\Omega_{hB}\left(\eta,\zeta;\eta',\zeta'\right):=\eta' \cdot\zeta - \eta\cdot \zeta' +h\widehat{B}_0 \zeta' \wedge \zeta.
$$
\end{proposition}

\begin{proof}
For $u \in L^2(\T^2,L)$, we have
\begin{multline*}
  W_h^B(\eta,\zeta) \circ W_h^B(\eta',\zeta') u(x) \\= e^{i  \eta \cdot( x+ \frac{h \zeta}{2} ) + \frac{ih \widehat{B}_0}{2} \zeta \wedge x} W_h^B(\eta',\zeta') u(x + h \zeta)\\
= e^{i \eta \cdot( x+ \frac{h \zeta}{2} )  + \frac{ih \widehat{B}_0}{2} \zeta \wedge x} e^{i  \eta' \cdot( x+h\zeta + \frac{h \zeta'}{2}) + \frac{ih\widehat{B}_0}{2} \zeta' \wedge (x+h \zeta)} u(x+h\zeta+h\zeta')\\
= e^{i\frac h 2 ( \eta' \cdot\zeta - \eta \cdot\zeta' + h \widehat{B}_0 \zeta' \wedge \zeta )}e^{i (\eta + \eta')\cdot(x+ \frac{h \zeta}{2} + \frac{h \zeta'}{2}) + \frac{ ih \widehat{B}_0}{2}(\zeta + \zeta') \wedge x}u(x+h\zeta+h\zeta') \\
=  e^{i\frac h 2 ( \eta' \cdot\zeta - \eta \cdot\zeta'  + h \widehat{B}_0 \zeta' \wedge \zeta )} W_h^B(\eta +\eta', \zeta+\zeta') u (x).
\end{multline*}
\end{proof}

Given a symbol $a  : \mathbf T^2 \times \mathbf R^2 \to \mathbf C$ which is periodic with respect to the first variable, we formally define the magnetic quantization using the Weyl operators $W_h^B$ by the following formula
\begin{equation} \label{eq.Opa}
 \Op (a)  = \frac{1}{(2\pi)^2} \sum_{\eta \in 2\pi  \Z^2} \int_{\R^2} \mathcal{F}(a)(\eta,\zeta) W_h^B (\eta,\zeta) \dd \zeta, 
\end{equation}
where we denoted by $\mathcal F$ the Fourier transform of the symbol $a(x,\xi)$, namely
\begin{equation}
\mathcal F(a)(\eta,\zeta) = \int_{\T^2 \times \R^2} a(x,\xi) e^{-i \eta\cdot x}e^{-i  \zeta \cdot\xi } \dd x \dd \xi.
\end{equation}
Note that the inverse formula for $\mathcal{F}$ is
\[\mathcal F^{-1}(a)(x,\xi) = \frac{1}{(2\pi)^2} \sum_{\eta \in 2\pi \Z^2} \int_{\R^2} a(\eta,\zeta) e^{i  x\cdot \eta } e^{i \xi\cdot \zeta } \dd \zeta.\]
Later on in the arguments, we shall use the short notation $w=(x,\xi)\in T^*\T^2$. This definition \eqref{eq.Opa} is analog to the standard Weyl quantization, which correspond to the case $\widehat{B}_0=0$. In that case, one can write \eqref{eq.Opa} in terms of the Fourier coefficients of $u$ -- see e.g.~\cite[Appendix]{MaciaRiviere2018}. However when $\widehat{B}_0 \neq 0$, $u$ is not periodic and these techniques no longer make sense. That is a reason why we work instead with the representation \eqref{eq.Opa}. This quantization is the analogue in a periodic setting of the one introduced by Mantoiu and Purice (see \cite{MantoiuPurice2004,IftimieMantoiuPurice2007,HelfferPurice2010, IftimieMantoiuPurice2019} for instance), and we hope that this article opens further potential applications of their formalism to spectral geometry.

\medskip
Before studying this quantization for general symbols, let us give a couple of important examples. First of all, if $a(x,\xi) = \xi_j$ (for $j=1,2$) then we obtain the magnetic derivative 
\begin{equation}\label{eq.opxi}
\Op(\xi_j) = -ih\partial_{x_j} - h A_j^{0},
\end{equation}
where we recall that $A^0(x_1,x_2) = \frac{\widehat{B}_0}{2}(-x_2, x_1)$ is our choice of gauge generating the uniform field of strength $\widehat{B}_0$. Recall that, even if $A^0$ is not periodic, this expression makes sense as an operator on $\mathscr{C}^{\infty}(\T^2,L)$. When $a$ is independent of $\xi$, we obtain the multiplication operator \begin{equation} \Op(a(x)) = a(x), \end{equation} as in the non-magnetic case. More generaly, when $a$ belongs to a suitable class of symbol, $\Op(a)$ defines a continous operator acting on the space $\mathscr{C}^{\infty}(\T^2,L)$ of smooth magnetic-periodic functions, introduced in \eqref{eq.CL}.

\begin{proposition}
Let $m \in \mathbf R$, and assume $a \in \mathscr{S}^m(T^*\T^2)$, where
\[ \mathscr{S}^m(T^*\T^2) = \lbrace a \in \mathscr{C}^{\infty}(\T^2 \times \R^2,\C), \forall \alpha, \beta \in \mathbf N_0^{2}, |\partial_x^{\alpha} \partial_{\xi}^{\beta} a (x,\xi)| \leq C_{\alpha \beta} \langle \xi \rangle^m \rbrace.\]
Then $\Op(a) : \mathscr{C}^{\infty}(\T^2,L) \rightarrow \mathscr{C}^{\infty}(\T^2,L) $ defines a continuous operator. Moreover, the formal adjoint is given by
\begin{equation}\label{eq.adjoint}
\Op(a)^* = \Op( \overline{a}).
\end{equation}
\end{proposition}
Observe that, in the following, we will deal symbols $(a_h)_{0<h\leq 1}$ and we will say that $(a_h)_{0<h\leq 1}$ is uniformly in $\mathscr{S}^m(T^*\T^2)$ or that $a_h=\mathscr{O}_{\mathscr{S}^m}(1)$ if all the semi-norms are uniformly bounded in terms of $0<h\leq 1$.
\begin{proof}
Here the magnetic field adds no significant problems, so we can use the usual strategy through oscillatory integral arguments~\cite[\S3.6]{Zworski}. For $u \in \mathscr{C}^{\infty}(\T^2,L)$, we formally integrate by parts in the definition of $\Op(a)$ until we get a convergent integral. These formal calculations correspond to interpreting \eqref{eq.Opa} in distributional sense as in~\cite[Th.~3.18]{Zworski}. Using the Fourier multiplier $\langle D_x \rangle = (1 - \partial^2_{x_1}-\partial^2_{x_2})^{\frac 12}$ and properties of the Fourier transform we have, for $\beta$, $\gamma$, $N \in 2\mathbf N$,
\begin{multline*}
\Op (a) u(x) = \frac{1}{(2\pi)^2} \sum_{\eta \in 2 \pi \mathbf Z^2} \langle \eta \rangle^{- \gamma} \\
\times\int_{\mathbf{R}^2} \langle D_{\zeta} \rangle^{N} \Big[ \langle \zeta \rangle^{- \beta} \mathcal{F}\big( \langle \xi \rangle^{-N}\langle D_y \rangle^{\gamma} \langle D_{\xi}\rangle ^{\beta} a \big) \Big]  W_h^B (\eta,\zeta) u(x) \dd \zeta.
\end{multline*}
When we integrate by part, $D_{\zeta}$ now acts on $W_h^B(\eta,\zeta) u (x)$,
\begin{multline*}
\Op (a) u(x) = \frac{1}{(2\pi)^2} \sum_{\eta \in 2 \pi \mathbf Z^2} \langle \eta \rangle^{- \gamma}\\ 
\times\int_{\mathbf{R}^2}  \langle \zeta \rangle^{- \beta} \mathcal{F}\big( \langle \xi \rangle^{-N}\langle D_y \rangle^{\gamma} \langle D_{\xi}\rangle ^{\beta} a \big)  \langle D_{\zeta} \rangle^{N} \Big[  W_h^B (\eta,\zeta) u(x)  \Big] \dd \zeta.
\end{multline*}
The derivatives with respect to $\zeta$ of the Weyl operator are bounded by seminorms of $u$ and powers of $\eta$,
\begin{multline*}
| \Op (a) u(x) | \leq C_{\gamma,\beta, N} \sum_{\eta \in 2 \pi \mathbf Z^2} \langle \eta \rangle^{- \gamma} \\
\times\int_{\mathbf{R}^2}  \langle \zeta \rangle^{- \beta} \big| \mathcal{F}\big( \langle \xi \rangle^{-N}\langle D_y \rangle^{\gamma} \langle D_{\xi}\rangle ^{\beta} a \big) \big| \langle \eta \rangle^{N} \| u \|_{\mathscr{C}^N} \dd \zeta.
\end{multline*}
If $a \in \mathscr{S}^m$ and $N$ large enough, then the Fourier transform $\mathcal F$ appearing in the previous integral is a bounded function. Therefore,
\begin{align*}
\big| \Op (a) u(x) \big| \leq C_{\gamma,\beta, N} \sum_{\eta \in 2 \pi \mathbf Z^2} \langle \eta \rangle^{ N - \gamma} \int_{\mathbf{R}^2}  \langle \zeta \rangle^{- \beta}  \dd \zeta \,  \| u \|_{\mathscr{C}^N},
\end{align*}
and the above sum and integral are convergent for $\gamma$ and $\beta$ large enough. Similarly, we obtain bounds on the derivatives of $\Op (a) u$, and deduce that $\Op (a)$ defines a continuous operator on $\mathscr{C}^{\infty}(\T^2,L)$. The formula for the adjoint follows from the definition \eqref{eq.Opa} and property \eqref{eq.Whstar}.
\end{proof}

\subsection{Composition formula}

We can now turn to the first main property of this quantization procedure.

\begin{theorem}\label{thm.comp}
Let $m_1$, $m_2 \in \mathbf R$ and assume $a \in \mathscr{S}^{m_1}$, $b \in \mathscr{S}^{m_2}$. Then,
\begin{enumerate}
\item There exists a $h$-dependent symbol $a \star b \in \mathscr{S}^{m_1 + m_2}$ such that
\[\Op(a) \circ \Op(b) = \Op(a \star b).\]
\item The symbol $a \star b$ is explicitly given by the oscillatory integral
\begin{align*}
a \star b(x,\xi) &= \sum_{\eta, \eta' \in 2 \pi \mathbf Z^2} \int_{\mathbf T^4} e^{i (\eta y + \eta' y' +2 \widehat{B}_0 y \wedge y')}a\Big(x - y, \xi + \frac{h \eta'}{2} + h \widehat{B}_0y^{\perp} \Big) \\ &\qquad \qquad \times  b \Big( x - y', \xi - \frac{h \eta}{2} + h \widehat{B}_0 y'^{\perp} \Big) \dd y \dd y',
\end{align*}
where $y^\perp=(-y_2,y_1).$
\item For all $N \geq 1$ we have the asymptotic expansion
\begin{multline*}
a \star b (x, \xi) = \sum_{n = 0}^{N-1}\frac{1}{n!} \Big( \frac{ih}{2} \Big)^n  \Omega_{hB}( D_x, D_\xi, D_{x'}, D_{\xi'})^n a(x,\xi) b(x',\xi') |_{(x,\xi) = (x',\xi')} \\
+ \mathscr{O}_{\mathscr{S}^{m_1 + m_2}}(h^N),
\end{multline*}
where all the seminorms in the remainder are bounded in terms of a finite number (depending only on the seminorm) of seminorms of $a$ and $b$.
\item In particular,
\[ a \star b = ab + \mathscr{O}_{\mathscr{S}^{m_1 +m_2}}(h), \]
and
\begin{equation}\label{eq.commutator}
 a \star b - b \star a = \frac{h}{i} \lbrace a,b \rbrace + \frac{\widehat{B}_0 h^2}{i} \big( \partial_{\xi_2} a \partial_{\xi_1} b - \partial_{\xi_1} a \partial_{\xi_2}b \big) + \mathscr{O}_{\mathscr{S}^{m_1 +m_2}}(h^3),
\end{equation}
where all the seminorms in the remainder are bounded in terms of a finite number (depending only on the seminorm) of seminorms of $a$ and $b$.
\item If $a(\xi)$ is a polynomial in $\xi$ of degree $\leq N$, then
$$
a \star b (x, \xi) = \sum_{n = 0}^{N} \frac{1}{n!}\Big( \frac{ih}{2} \Big)^n  \Omega_{hB}( D_x, D_\xi, D_{x'}, D_{\xi'})^n a(x,\xi) b(x',\xi') |_{(x,\xi) = (x',\xi')} 
$$
The same holds if $a(x,\xi)$ and $b(x,\xi)$ are both polynomial in $\xi$ with total degree $\leq N$.
\end{enumerate}
\end{theorem}

\begin{proof}
Using the composition rule for Weyl operators (Proposition \ref{prop.compT}), we find
\begin{multline*}
\Op(a) \circ \Op(b) = \frac{1}{(2\pi)^4} \sum_{\eta, \eta' \in  2\pi\Z^2} \int_{\R^4} e^{i\frac h 2 ( \eta' \cdot\zeta  - \eta \cdot\zeta')} e^{\frac{i h^2}{2} \widehat{B}_0 \zeta' \wedge \zeta } \\  \times \mathcal{F}(a)(\eta, \zeta) \mathcal{F}(b)(\eta',\zeta')  W_h^B(\eta +\eta', \zeta+\zeta') \dd \zeta \dd \zeta'.
\end{multline*}
We translate the variables $\eta$ and $\zeta$ to get
\begin{multline*}
\Op(a) \circ \Op(b) = \frac{1}{(2\pi)^4} \sum_{\eta, \eta' \in  2\pi\Z^2} \int_{\R^4}   e^{i\frac h 2 (  \eta' \cdot( \zeta - \zeta' ) - ( \eta-\eta' ) \cdot\zeta' )} e^{\frac{i h^2}{2} \widehat{B}_0 \zeta' \wedge \zeta } \\ \times \mathcal{F}(a)(\eta -\eta', \zeta-\zeta') \mathcal{F}(b)(\eta',\zeta') W_h^B(\eta , \zeta) \dd \zeta \dd \zeta'.
\end{multline*}
Thus $\Op(a) \circ \Op(b) = \Op(a \star b)$ with
\begin{align}\label{eq.Fab}
\mathcal{F} (a \star b) (\eta,\zeta) = \frac{1}{(2\pi)^2} \sum_{\eta' \in  2\pi\Z^2} \int_{\R^2} &e^{i\frac h 2 (  \eta' \cdot( \zeta - \zeta' ) - ( \eta-\eta' )\cdot \zeta' )} e^{\frac{i h^2}{2} \widehat{B}_0 \zeta' \wedge \zeta } \\ & \quad \times \mathcal{F}(a)(\eta- \eta', \zeta-\zeta') \mathcal{F}(b)(\eta',\zeta') \dd \zeta'. \nonumber
\end{align}
Inverting the Fourier transform, we find
\begin{align*}
a \star b (x,\xi) = \frac{1}{(2\pi)^4} \sum_{\eta,\eta' \in 2\pi \Z^2} \int_{\R^4} &e^{i \eta\cdot x} e^{i \zeta\cdot \xi} e^{i\frac h 2 (  \eta' \cdot( \zeta - \zeta' ) - ( \eta-\eta' ) \cdot\zeta' )} e^{\frac{i h^2}{2} \widehat{B}_0 \zeta' \wedge \zeta}\\ & \quad \times \mathcal{F}(a)(\eta- \eta', \zeta-\zeta') \mathcal{F}(b)(\eta',\zeta') \dd \zeta \dd \zeta'. 
\end{align*}
We translate $\eta$ and $\zeta$ back to obtain
\begin{align*}
a \star b (x,\xi) = \frac{1}{(2\pi)^4} \sum_{\eta,\eta' \in 2\pi \Z^2} \int_{\R^4} &e^{i (\eta + \eta')\cdot x} e^{i (\zeta +\zeta') \cdot\xi} e^{i\frac h 2 (  \eta'\cdot \zeta - \eta \cdot\zeta' )} e^{\frac{i h^2}{2} \widehat{B}_0 \zeta' \wedge \zeta}\\ & \quad \times \mathcal{F}(a)(\eta, \zeta) \mathcal{F}(b)(\eta',\zeta') \dd \zeta \dd \zeta'.
\end{align*}
We recall the notation $\Omega_{hB}(\eta,\zeta;\eta' , \zeta') = \eta' \zeta - \eta \zeta' + h \widehat{B}_0 \zeta' \wedge \zeta$, and we can insert the definition of the Fourier transform for $a$ and $b$,
\begin{align*}
a \star b (x,\xi) &= \sum_{\eta,\eta' \in 2\pi \Z^2} \int_{\R^4}\int_{(\mathbf T^2 \times \mathbf R^2 )^2} e^{i (x,\xi)\cdot(\eta, \zeta)} e^{i (x,\xi)\cdot(\eta', \zeta') } e^{i\frac h 2\Omega_{hB}(\eta,\zeta,\eta' , \zeta')}  \\ & \qquad \qquad \qquad \qquad \times a(w) b(w') e^{-i  w \cdot(\eta,\zeta) }e^{-i  w'\cdot(\eta', \zeta')} \frac{ \dd w \dd w' \dd \zeta \dd \zeta'}{(2\pi)^4}.
\end{align*}
Now we change $w$ and $w'$ by $z-w$ and $z-w'$ with $z=(x,\xi)$ to obtain
\begin{align} \nonumber
a \star b (z) &= \frac{1}{(2\pi)^4} \sum_{\eta,\eta' \in 2\pi \Z^2} \int_{\R^4}\int_{(\mathbf T^2 \times \mathbf R^2 )^2}  e^{i\frac h 2\Omega_{hB}(\eta,\zeta;\eta' , \zeta')} e^{i w \cdot (\eta,\zeta)}e^{i w'\cdot (\eta', \zeta')}  \\ & \qquad \qquad \qquad \qquad \times a(z-w) b(z-w')  \dd w \dd w' \dd \zeta \dd \zeta' \nonumber \\
&= \int_{(\mathbf T^2 \times \mathbf R^2 )^2} G_h(w,w') a(z-w) b(z-w') \dd w \dd w', \label{eq.ab0}
\end{align}
where
\begin{align*}
G_h(w,w') =  \frac{1}{(2\pi)^4} \sum_{\eta,\eta' \in 2\pi \Z^2} \int_{\R^4} e^{i\frac h 2\Omega_{hB}(\eta,\zeta;\eta' , \zeta')} e^{i w \cdot (\eta,\zeta)}e^{i w'\cdot (\eta', \zeta')}\dd \zeta \dd \zeta',
\end{align*}
is the inverse Fourier transform of $\exp (i \frac h 2 \Omega_{hB})$. We can compute $G_h$ explicitly as a sum of $\delta$ measures. Indeed, using the notations $w=(y,\theta)$ and $w'=(y',\theta')$, we have
\begin{align*}
G_h(w,w') =  \sum_{\eta, \eta' \in 2\pi \mathbf Z^2} \int_{\mathbf R^4} e^{\frac{i h^2}{2} \widehat{B}_0 \zeta' \wedge \zeta }e^{\frac{ih}{2}(\eta'\cdot \zeta - \eta \cdot\zeta')} e^{i (y\cdot \eta + \theta \cdot\zeta + y'\cdot \eta' + \theta' \cdot\zeta') }\frac{ \dd \zeta \dd \zeta'}{(2\pi)^4}. 
\end{align*}
Recall now that the Dirac distribution $\delta$ on $\R^d$ verifies $\delta(\lambda x)=\lambda^{-d}\delta(x)$ and that one has the Poisson formula
$$
\sum_{m\in\Z^d}\delta(x-m)=\sum_{\eta\in 2\pi \Z^d}e^{i\eta\cdot x}\ \Rightarrow\ \forall\lambda>0,\ \sum_{m\in\Z^d}\delta\left(x-\frac{m}{\lambda}\right)=\lambda^d\sum_{\eta\in 2\pi \Z^d}e^{i\eta\cdot \lambda x}.
$$
Implementing this formula (for $d=4$) in the above expression of $G_h$, one finds
\begin{align*}
G_h(w,w') = \frac{1}{(\pi h)^4} \sum_{m,m' \in \mathbf Z^2} e^{2i \widehat{B}_0 (y' + m')\wedge (y + m)} e^{\frac{2i}{h} \theta' (y+m)} e^{-\frac{2i}{h} \theta (y'+m')}.
\end{align*}
Since $\widehat{B}_0 m' \wedge m \in 2 \pi\mathbf Z$, we deduce
\begin{multline*}
G_h(w,w') = \frac{1}{(\pi h)^4} e^{2i \widehat{B}_0 y' \wedge y } e^{\frac{2i}{h} ( \theta' y - \theta y') }
\\
\times\sum_{m \in \mathbf Z^2}  e^{im \big( \frac{2}{h} \theta' + 2 \widehat{B}_0 y'^{\perp} \big)} \sum_{m' \in \mathbf Z^2} e^{im' \big(- \frac{2}{h} \theta - 2 \widehat{B}_0 y^{\perp} \big)} ,
\end{multline*}
where $y^\perp = (-y_2,y_1)$ is such that $y \wedge m = m \cdot y^{\perp}$. Using the Poisson formula again we find
\begin{multline} \label{eq.Gh}
G_h(w,w') = e^{-2i \widehat{B}_0 y' \wedge y} \sum_{\eta \in 2\pi \mathbf Z^2} e^{i\eta y} \delta_{\lbrace \theta' = -h \widehat{B}_0 y'^{\perp} + \frac{h \eta}{2} \rbrace} \\
\times\sum_{\eta' \in 2\pi \mathbf Z^2} e^{i\eta' y'} \delta_{\lbrace \theta = -h \widehat{B}_0 y^{\perp} - \frac{h \eta'}{2} \rbrace}.
\end{multline}
We insert this in \eqref{eq.ab0} to find the product formula,
\begin{align*}
a \star b(x,\xi) &=  \sum_{\eta, \eta' \in 2 \pi \mathbf Z^2} \int_{\mathbf T^4} e^{i (\eta y + \eta' y' +2 \widehat{B}_0 y \wedge y')}  a\Big(x - y, \xi + \frac{h \eta'}{2} + h \widehat{B}_0 y^{\perp}\Big)
\\ &\qquad \qquad \times b \Big( x - y', \xi - \frac{h \eta}{2} + h \widehat{B}_0 y'^{\perp} \Big) \dd y \dd y'.
\end{align*}
Let us now argue that $a \star b$ belongs to the class $\mathscr{S}^{m_1 + m_2}$. We use the Fourier multiplier $\langle D_y \rangle$, integrate by part, and use that $a \in \mathscr{S}^{m_1}$ and $b \in \mathscr{S}^{m_2}$,
\begin{align*}
| a \star b(x,\xi)| &= \Big| \sum_{\eta, \eta' \in 2 \pi \mathbf Z^2} \int_{\mathbf T^4} \langle \eta \rangle^{-N}\langle \eta' \rangle^{-N}\langle D_y \rangle^N \langle D_{y'} \rangle^N \Big( e^{i (\eta y + \eta' y')} \Big)\\ &\times e^{2i \widehat{B}_0 y \wedge y'} a\big(x - y, \xi + \frac{h \eta'}{2} + h \widehat{B}_0 y^\perp \big) b \big( x - y', \xi - \frac{h \eta}{2} + h \widehat{B}_0 y'^{\perp} \big) \dd y \dd y' \Big| \\
&= \Big|\sum_{\eta, \eta' \in 2 \pi \mathbf Z^2} \langle \eta \rangle^{-N}\langle \eta' \rangle^{-N} \int_{\mathbf T^4}   e^{i (\eta y + \eta' y')}  \langle D_y \rangle^N \langle D_{y'} \rangle^N \Big( e^{2i \widehat{B}_0 y \wedge y'} \\ &\times a\big(x - y, \xi + \frac{h \eta'}{2} +h \widehat{B}_0 y^{\perp} \big) b \big( x - y', \xi - \frac{h \eta}{2} + h \widehat{B}_0 y'^{\perp} \big) \Big) \dd y \dd y' \Big| \\
& \leq C \sum_{\eta, \eta' \in 2 \pi \mathbf Z^2} \langle \eta \rangle^{-N}\langle \eta' \rangle^{-N}\\
&´\times\int_{\mathbf T^4} \langle \xi + \frac{h \eta'}{2} + h \widehat{B}_0 y^{\perp} \rangle^{m_1} \langle \xi - \frac{h \eta}{2} + h \widehat{B}_0y'^{\perp} \rangle^{m_2} \dd y \dd y'\\
& \leq C \langle \xi \rangle^{m_1 + m_2},
\end{align*}
for $N$ is large enough. We can use the same argument to bound derivatives of $a\star b$ and we deduce that $a \star b \in \mathscr{S}^{m_1+m_2}$. Here, all the seminorms of $a\star b$ are bilinearly bounded by a finite number (independent of the choice of $a$ and $b$) of semi-norms of $a$ and $b$. To obtain the expansion of $a \star b$, we start from \eqref{eq.Fab}, and use a Taylor expansion at order $N$ of the exponential:
\begin{align*}
\mathcal{F} (a \star b) (\eta,\zeta) &= \frac{1}{(2\pi)^2} \sum_{\eta' \in  2\pi\Z^2} \int_{\R^2} e^{i\frac h 2\Omega_{hB}( \eta - \eta', \zeta - \zeta', \eta', \zeta')} \\
&\qquad\qquad\times\mathcal{F}(a)(\eta- \eta', \zeta-\zeta') \mathcal{F}(b)(\eta',\zeta') \dd \zeta' \\
&= \frac{1}{(2\pi)^2} \sum_{n = 0}^{N-1} \frac{1}{n!}\Big( \frac{ih}{2} \Big)^n  \sum_{\eta' \in  2\pi\Z^2} \int_{\R^2} \Omega_{hB}( \eta - \eta', \zeta - \zeta', \eta', \zeta')^n\\ &\qquad \qquad \qquad \times \mathcal{F}(a)(\eta- \eta', \zeta-\zeta') \mathcal{F}(b)(\eta',\zeta') \dd \zeta' + h^N \mathcal F (\mathcal R_N ).
\end{align*}
We recognize the Fourier multiplier
\begin{multline*}
\mathcal{F} (a \star b) (\eta,\zeta) =\\ \sum_{n = 0}^{N-1}\frac{1}{n!} \Big( \frac{ih}{2} \Big)^n \mathcal F \Big( \Omega_{hB}( D_x, D_\xi, D_{x'}, D_{\xi'})^n a(x,\xi) b(x',\xi') |_{(x,\xi) = (x',\xi')} \Big) \\ + h^N \mathcal{F}(\mathcal R_N).
\end{multline*}
We can use the inverse Fourier transform and control the remainder term $\mathcal R_N$ in the same way as above to deduce
\begin{multline*}
a \star b (x, \xi) = \sum_{n = 0}^{N-1} \frac{1}{n!}\Big( \frac{ih}{2} \Big)^n  \Omega_{hB}( D_x, D_\xi, D_{x'}, D_{\xi'})^n a(x,\xi) b(x',\xi') |_{(x,\xi) = (x',\xi')} \\
+ \mathscr{O}_{\mathscr{S}^{m_1 + m_2}}(h^N).
\end{multline*}
Observe that the remainder can be expressed in terms of a finite number of derivatives of $a$ and $b$ of order at least $N$ so that the involved seminorms in the remainder are bounded in terms of the seminorms of $a$ and $b$.
Note also that if $a(\xi)$ is a polynomial in $\xi$ of degree $\leq N-1$, then the remainder coming from the Taylor formula cancels out so that there is no remainder term. The same holds if both $a(x,\xi)$ and $b(x,\xi)$ are polynomial in $\xi$ of total degree $\leq N-1$. Moreover, the terms $n=1$ and $n=2$ are of the form
\begin{equation*}
(n=1) \qquad \frac{ih}{2} \big( \partial_x a \partial_\xi b - \partial_\xi a \partial_x b \big) + \frac{ih^2}{2}  \widehat{B}_0 \big( \partial_{\xi_1} a \partial_{\xi_2} b - \partial_{\xi_2} a \partial_{\xi_1} b \big),
\end{equation*}
\begin{equation*}
(n=2) \qquad -\frac{h^2}{8} \big( \partial_{\xi}^2 a \partial_x^2 b + \partial_x^2 a \partial_\xi^2 b - 2 \partial_x \partial_\xi a \partial_x \partial_\xi b\big) + \mathscr{O}_{\mathscr{S}^{m_1+m_2}}(h^3),
\end{equation*}
and the theorem follows. Note that, for $n=2$, the leading term is the same for $a\star b$ and $b\star a$.
\end{proof}

\subsection{Application to the magnetic Laplacian}

From the basic properties of the magnetic quantization, we see that
\[ -ih\partial_{x_j} - h \alpha_j - h A_j = \Op \big( \xi_j - h \alpha_j - h A^{\rm{per}}_j \big),\]
for $j=1,2$, and where $A=A^0 + A^{\rm{per}}$ is given by Lemma \ref{l:potential}.
Using the composition formula (Theorem \ref{thm.comp}) we deduce that $h^2 \mathscr{L}_\alpha$ is given by
\begin{align}
h^2\mathscr{L}_\alpha &=\sum_{j=1}^2\left(\Op\left(\xi_j-h\alpha_j-hA_j^{\rm per}\right)\right)^2 \\
&=\Op\left(|\xi - h \alpha - h A^{\rm{per}} |^2 \right). \nonumber
\end{align}
Here, we used that the symbols are linear in the $\xi$-variable so that the last item of Theorem~\ref{thm.comp} applies. In particular, for $j=1,2$, one has  
$$\left(\Op\left(\xi_j-h\alpha_j-hA_j^{\rm per}\right)\right)^2=\Op\left(\left(\xi_j-h\alpha_j-hA_j^{\rm per}\right)^2\right).
$$
Therefore, we can expan the symbol in powers of $h$ and we find
\begin{equation}\label{e:symbol-laplacian}
h^2\mathscr{L}_\alpha =\Op\left(|\xi |^2+hc_1(x)\cdot\xi+h^2c_0(x)\right),
\end{equation}
where $c_0$ and $c_1$ belong to $\mathscr{C}^\infty(\T^2,\R)$ and $\mathscr{C}^\infty(\T^2,\C^2)$ respectively, with explicit formulas:
$$
c_0(x)=\left|\alpha+A^{\rm per}(x)\right|^2\ \text{and}\ c_1(x)=-2(\alpha+A^{\rm per}(x)).
$$
In view of proving Theorem~\ref{t:maintheo-quantitative}, it is also worth collecting the following consequence of the composition formula.
\begin{corollary}\label{coro.comp.quadratic} Let $a\in \mathscr{S}^m$ and $\alpha\in\R^2$. Then, one has
\begin{multline*}
\Big[\Op\left(|\xi-h\alpha|^2\right),\Op(a)\Big] \\=\Op \Big(\frac{2h}{i}(\xi-h\alpha)\cdot\partial_xa-\frac{2h^2\widehat{B}_0}{i}(\xi-h\alpha)^\perp\cdot\partial_\xi a\Big).
\end{multline*}
\end{corollary}
The important observation here is that this is an exact formula. In particular, there is no remainder term (even $\mathscr{O}(h^\infty)$). This is due to the fact that $|\xi-h\alpha|^2$ is quadratic in $\xi$ and independent of $x$, which implies that the last item of Theorem~\ref{thm.comp} applies. 

\medskip
For later purpose, it is useful to compute the flow induced by the vector field appearing in this formula, namely
$$
\mathbf{X}_h(x,\xi):=2(\xi-h\alpha)\cdot\partial_x-2h\widehat{B}_0(\xi-h\alpha)^\perp\cdot\partial_\xi.
$$
An explicit computation shows that the induced flow is given by
\begin{equation}\label{e:magnetic-periodicflow}
\Phi_h^t\left(x,\xi\right)=\left(x+\frac{1}{h\widehat{B}_0}\left(e^{2t h\widehat{B}_0 J}-\text{Id}\right)(\xi-h\alpha)^{\perp},h\alpha+e^{2th \widehat{B}_0 J}(\xi-h\alpha)\right),
\end{equation}
where $J:=\left(\begin{array}{cc} 0 & 1\\ -1 &0\end{array}\right)$. From this expression, one can remark on the one hand that 
$$\lim_{h\rightarrow 0}\Phi_h^t(x,\xi)=(x+2t\xi,\xi)=\varphi^{2t}(x,\xi)$$ 
is the standard geodesic flow (up to reparametrization of time by a factor $2$). On the other hand, the flow $\Phi_h^t$ is $\pi/(h\widehat{B}_0)$-periodic which is somehow a manifestation of the cyclotron motion at work in magnetic fields. In order to understand this flow, one can first observe that $|\xi-h\alpha|$ is preserved under the flow $\Phi_h^t$ so that we can set $(\cos\theta,\sin\theta)=\xi-h\alpha$ as a new variable lying on a circle radius $1$ and the differential equation becomes 
$$
(x',\theta')=2(\cos\theta,\sin\theta,-2h\widehat{B}_0).
$$
It can be integrated as $$\theta=\theta_0-2th\widehat{B}_0\ \text{mod}\  2\pi,
$$
and
$$
\begin{cases}
x_1(t)=x_1(0)+2\int_0^{t}\cos(\theta_0-2\tau h\widehat{B}_0)d\tau,\ \\ x_2(t)=x_2(0)+2\int_0^{t}\sin(\theta_0-2\tau h\widehat{B}_0)d\tau.
\end{cases}
$$

\subsection{Calder\'on-Vaillancourt Theorem}

In view of applying this composition formula in concrete situations, one needs to ensure the boundedness of pseudo-differential operators on $L^2$ spaces. This is the content of the Calder\'on-Vaillancourt Theorem which can be stated as follows.

\begin{theorem}\label{thm.cv}
If $a \in \mathscr{S}^0$ then $\Op(a)$ is a bounded operator on $L^2(\T^2,L)$ and
for $h \in (0,1]$, one has
$$
\left\|\mathsf{Op}_h^B(a)\right\|_{L^2\rightarrow L^2}\leq C_0\sum_{|\gamma|+|\beta|\leq N_0}h^{|\beta|}\|\partial_x^\gamma\partial_\xi^\beta a\|_{\infty}.
$$
\end{theorem}

\begin{remark}\label{r:rescale}
 It is useful to observe that
\begin{equation}\label{e:change-semiclassical-parameter}
 \Op(a(x,\delta\xi))= \mathsf{Op}_{h\delta}^B(a(x,\xi)).
\end{equation}
In particular, we can deduce from the upper bound
$$
\left\|\mathsf{Op}_1^B(a)\right\|_{L^2\rightarrow L^2}\leq C_0\sum_{|\gamma|+|\beta|\leq N_0}\|\partial_x^\gamma\partial_\xi^\beta a\|_{\infty}
$$
that
$$
\left\|\mathsf{Op}_h^B(a)\right\|_{L^2\rightarrow L^2}\leq C_0\sum_{|\gamma|+|\beta|\leq N_0}h^{|\beta|}\|\partial_x^\gamma\partial_\xi^\beta a\|_{\infty},
$$
which will be useful in view of picking symbols oscillating along the $\xi$-variable.
\end{remark}

\begin{proof}
We follow the strategy from the proof of the Calder\'on-Vaillancourt Theorem given in \cite[Theorem 4.23]{Zworski}. Let $\chi$ be a smooth compactly supported cutoff function such that $0 \leq \chi \leq 1$ and
\begin{equation*}
\sum_{k \in \mathbf Z^2} \chi_k(\xi) =1, \qquad \chi_k(\xi) = \chi(\xi-k), \qquad \forall \xi \in \R^2.
\end{equation*}
Then we decompose $a$ as a sum of compactly supported symbols,
\begin{equation*}
a = \sum_{k \in \mathbf Z^2} a_k, \qquad a_k(x,\xi) = a(x,\xi) \chi_k(\xi).
\end{equation*}
For $k$, $\ell \in \mathbf Z^2$ we introduce the symbol $b_{k \ell} = \overline{a}_k \star a_\ell$. Let us first explain why the operators $\Opu (b_{k \ell})$ are bounded in $L^2$, with uniform bounds of the form
\begin{equation}
\| \Opu( b_{k \ell}) \|_{L^2 \to L^2} \leq C_N \langle k - \ell \rangle^{-N},
\end{equation}
for all $N > 0$. Using the composition formula in Theorem \ref{thm.comp}, and using the Fourier multipliers $\langle D_y \rangle$, $\langle D_{y'}\rangle$ we have
\begin{align*}
b_{k\ell}(x,\xi) &=  \int_{\mathbf T^4} \sum_{\eta, \eta' \in 2 \pi \mathbf Z^2} \langle \eta \rangle^{-M} \langle \eta' \rangle^{-M} e^{i (\eta y + \eta' y')} \langle D_y \rangle^M \langle D_{y'} \rangle^M \Big( e^{ 2i \widehat{B}_0 y \wedge y'} \\ &\qquad \qquad \times \overline a_k\big(x - y, \xi + \frac{ \eta'}{2} + \widehat{B}_0 y^{\perp} \big) a_\ell \big( x - y', \xi - \frac{ \eta}{2} + \widehat{B}_0 y'^{\perp} \big) \Big) \dd y \dd y'.
\end{align*}
We choose the fundamental domain $(0,1)^2$ to integrate, and bound the symbol as
\begin{align*}
|b_{k\ell}(x,\xi)| &\leq \sum_{\eta, \eta'} \langle \eta \rangle^{-M} \langle \eta' \rangle^{-M} \int_{(0,1)^4} \big| \langle D_y \rangle^M \langle D_{y'} \rangle^M \Big( e^{2i \widehat{B}_0 y \wedge y'} \\ &\qquad \qquad \times \overline a_k\big(x - y, \xi + \frac{ \eta'}{2} + \widehat{B}_0 y^{\perp} \big) a_\ell \big( x - y', \xi - \frac{ \eta}{2} +\widehat{B}_0 y'^{\perp} \big) \Big) \big|\dd y \dd y'.
\end{align*}
Due to the support of $a_k$ and $a_\ell$, the integrand is vanishing unless
\begin{align*}
| \xi + \frac{ \eta'}{2} + \widehat{B}_0 y^{\perp} - k | \leq C, \quad \text{and} \quad | \xi - \frac{ \eta}{2} +  \widehat{B}_0 y'^{\perp} - \ell | \leq C,
\end{align*}
which implies
\begin{equation*}
\langle k - \ell \rangle \leq C \langle \eta \rangle + C \langle \eta' \rangle, \quad \text{and} \quad \langle \xi - \frac{k + \ell}{2} \rangle \leq C \langle \eta \rangle+ C \langle \eta' \rangle,
\end{equation*}
where the constant $C$ may depend on $\widehat{B}_0$.  Therefore,
\begin{align*}
|b_{k\ell}(x,\xi)|  \leq C \sum_{\eta, \eta'} \langle \eta \rangle^{2N -M} \langle \eta' \rangle^{2N-M} \langle k - \ell \rangle^{-N} \langle \xi - \frac{k + \ell}{2} \rangle^{-N},
\end{align*}
where the constant now depends on at most $2M$ derivatives of $a$ (on top of its dependence on $\widehat{B}_0$, $N$ and $M$). Taking $M$ large enough (compared with $N$) the sum converges, and arguing similarly for derivatives of $b$ we obtain
\begin{equation}
| \partial_x^\alpha \partial_{\xi}^{\beta} b(x,\xi) | \leq C_N \langle k - \ell \rangle^{-N} \langle \xi - \frac{k + \ell}{2} \rangle^{-N}.
\end{equation}
We now use the definition \eqref{eq.Opa} of the quantization to find that, for a given $N>0$, one can pick $M>0$ large enough to derive
\begin{align*}
\| \Opu (b_{k \ell}) \|_{L^2 \to L^2} &\leq \frac{1}{(2\pi)^2} \sum_{\eta \in 2\pi \Z^2} \int_{\R^2} | \mathcal F(b_{k\ell})(\eta,\zeta) | \dd \zeta \\
& \leq C \sum_{\eta \in 2\pi \Z^2} \int_{\R^2} \langle \eta \rangle^{-M} \langle \zeta \rangle^{-M}\big| \mathcal F \big( \langle D_x \rangle ^M \langle D_\xi \rangle^M b_{k\ell} \big) \big| \dd \zeta \\
&\leq C \sum_{\eta \in 2\pi \Z^2} \int_{\R^2} \langle \eta \rangle^{-M} \langle \zeta \rangle^{-M} \| \langle D_x \rangle ^M \langle D_\xi \rangle^M b_{k\ell}  \|_{L^1} \dd \zeta\\
& \leq C \langle k - \ell \rangle^{-N},
\end{align*}
and note that $C$ only depends on a finite number of derivatives of $a$. Since 
$$\Opu( b_{k\ell}) = \Opu(a_k)^* \Opu(a_\ell),$$ 
we deduce with $N$ large enough that 
\begin{equation}
\sup_k \sum_{\ell} \| \Opu(a_k)^* \Opu(a_\ell) \|^{1/2} \leq C(a),
\end{equation}
for some $C(a)$ independent of $k$, $l$, but depending on finitely many derivatives of $a$. Similarly,
\begin{equation}
\sup_k \sum_{\ell} \| \Opu(a_k)\Opu(a_\ell)^* \|^{1/2} \leq C(a),
\end{equation}
and it follows from the Cotlar-Stein Theorem \cite[C.5]{Zworski} that $\Opu(a) = \sum_k \Opu(a_k)$ is a bounded operator on $L^2$, with norm smaller than $C(a)$. The statement for $h \in(0,1]$ follows from Remark \ref{r:rescale}.
\end{proof}

\begin{remark}
In Section~\ref{s:periodic}, we will have to study multiscale properties of magnetic eigenfunctions. To that aim, we will consider observables of the form $a\left(x,\xi,\frac{\xi}{h^{\frac12}}\right)$ and we will use repeatedly Remark~\ref{r:rescale} under the following form:
$$
\Op\left(a\left(x,\xi,\frac{\xi}{h^{\frac12}}\right)\right)=\mathsf{Op}_{h^{\frac12}}^B\left(a\left(x,h^{\frac12}\xi,\xi\right)\right).
$$
In the $h^{\frac12}$-calculus, the symbol is obviously uniformly in $\mathscr{S}^0(T^*\T^2)$  as soon as $a$ is compactly supported. Moreover, as $h^{\frac12}\xi$ is bounded, one has that the symbol $ h^{\frac m2}a\left(x,h^{\frac12}\xi,\xi\right)$ is uniformly in $\mathscr{S}^{-m}(T^*\T^2)$.
\end{remark}

We finally deduce from Theorems \ref{thm.comp} and \ref{thm.cv} the following useful properties of our quantization procedure.

\begin{corollary}
Let $a \in \mathscr{S}^m$, $b \in \mathscr{S}^{-m}$ with $m\in\R$. Then one has the following properties.
\begin{enumerate}
\item $\| \Op(a) \Op(b) - \Op(ab) \|_{L^2\rightarrow L^2} =\mathscr{O}(h)$,
where the constant in the remainder is bounded in terms of a finite number of seminorms of $a$.
\item at order $3$ 
\begin{multline*}\left\| \big[ \Op(a) , \Op(b) \big] - \frac{h}{i} \Op\big( \lbrace a,b \rbrace \big)-\frac{\widehat{B}_0h^2}{i}\Op\big(\partial_{\xi_2}a\partial_{\xi_1}b-\partial_{\xi_1}a\partial_{\xi_2}b\big) \right\|_{L^2\rightarrow L^2}\\ = \mathscr{O}(h^3),
 \end{multline*}
where the constant in the remainder is bounded in terms of a finite number of seminorms of $a$.
\item If ${\rm{supp}} \, a \cap {\rm{supp}}\, b = \emptyset$ then $\| \Op(a) \Op(b) \|_{L^2\rightarrow L^2} = \mathscr{O}(h^{\infty})$, where the constant in the remainder is bounded in terms of a finite number of seminorms of $a$ and $b$.
\end{enumerate}
\end{corollary}

\section{Proof for constant magnetic fields} \label{s:proof-quantitative}

 In this Section, we will give a proof of Theorem~\ref{t:maintheo-quantitative}, meaning that we will deal with the simplified case where $B$ is constant equal to $\widehat{B}_0$ and where $V$ is identically $0$. In that case, we are able to give a more direct proof of our quantum unique ergodicity property based on the exact formula given in Corollary~\ref{coro.comp.quadratic} together with appropriate stationary phase estimates. This allows to illustrate the dynamical mechanisms at work behind our proof even if it will not appear as explicitly when dealing with more general magnetic fields. Note that our proof does not make use of the exact expression of the eigenfunctions derived in Section~\ref{s:laplacian}. 

All along this section, $a$ will denote an element in $\mathscr{S}^0$ and we will decompose it in Fourier series along the $x$ variable as follows:
$$
a(x,\xi):=\sum_{k\in\mathbb{Z}^2}\widehat{a}_k(\xi)e^{2i\pi k\cdot x}.
$$
We also fix a function $\chi\in\mathscr{C}^\infty(\R,\R_+)$ which is identically equal to $1$ (resp. $0$) on $[-1/4,1/4]$ (resp. outside $[-1/2,1/2]$). From $\chi$ and given any $0<\delta<1$, we define the smooth function
$$
\chi_\delta(\xi)=\chi\left(\frac{|\xi-h\alpha|-1}{\delta}\right).
$$

\subsection{Averaging method}\label{ss:egorov-exact}

The proof of Theorem \ref{t:maintheo-quantitative} is divided into two steps. The first and main one is the following averaging lemma, which tells us that, to analyze the limits of Wigner distributions of the form $\langle \Op(a) u_h,u_h \rangle$, we can reduce to the case when $a$ depends only on $|\xi -h \alpha|$. In the second step (Section \ref{sec.microloc}), we will show that in fact $u_h$ is microlocalized where $|\xi | \sim | \xi -h\alpha | =1$, and use it to remove the last $\xi$ dependence.

\begin{lemma}\label{l:exact-egorov} Suppose that $\widehat{B}_0\in 2\pi\Z^*$. There exist constants $C_0,N_0>0$ such that, for any $0<h<1$, any $a\in \mathscr{S}^0$ and any $u_h$ solution to
$$
h^2\mathscr{L}_\alpha u_h=u_h,\quad\|u_h\|_{L^2}=1,
$$
one has
\begin{multline*}
\left|\langle \Op(a)u_h,u_h\rangle-\left\langle \Op\left(\chi_1(\xi)\frac{1}{2\pi}\int_{\mathbf S^1 }\widehat{a}_0 \big(|\xi-h\alpha|\theta+h\alpha\big)\dd\theta\right)u_h,u_h\right\rangle\right|\\
\leq C_0\|a\|_{\mathscr{C}^{N_0}}h^{\frac12}.
\end{multline*}
\end{lemma}

The key idea behind this Lemma is to use the periodic properties of the underlying magnetic flow $\Phi^t$, which is given by \eqref{e:magnetic-periodicflow}, and is associated to the vector field
\[
\mathbf{X}_h(x,\xi):=2(\xi-h\alpha)\cdot\partial_x-2h\widehat{B}_0(\xi-h\alpha)^\perp\cdot\partial_\xi.
 \]
In fact, from the eigenvalue equation we deduce the commutator estimate
\begin{equation}\label{eq.commutator14}
 \langle \left[ h^2 \mathscr{L}_h, \Op(a) \right] u_h, u_h \rangle = 0 ,
 \end{equation}
which tells us that eigenfunctions are stationary states. The fundamental properties of the magnetic Weyl quantization imply that such a commutator is the quantization of a Poisson bracket. More precisely, as noted in Corollary \ref{coro.comp.quadratic}, equation \eqref{eq.commutator14} can be rewritten as
\begin{equation}\label{eq.average14}
\langle \Op \left( \mathbf X_h (a) \right) u_h, u_h \rangle = 0,
\end{equation}
which now gives invariance by the \emph{classical} flow $\Phi^t$. We can therefore replace $a$ by its average along the flow. In fact the flow $\Phi^t$, which is nothing else that the cyclotron motion, is periodic with period $\pi/(h\widehat{B}_0)$. It follows very large circles, if we see the torus as a quotient of $\R^2$. This periodicity allows us to apply an averaging argument similar to the one used by Weinstein in his seminal work on eigenvalues of Schr\"odinger operators on the sphere~\cite{Weinstein1977}. See also~\cite{MaciaRiviere2016, MaciaRiviere2019, ArnaizMacia2022a, ArnaizMacia2022, CharlesLefeuvre2024} for similar arguments in the context of semiclassical measures. The extra difficulty compared with these references where the  flow is $2\pi$-periodic, is that we average over very long semiclassical times, of order $h^{-1}$. Yet, since the invariance \eqref{eq.average14} is \emph{exact} (with no remainder), the averaging procedure makes sense even after very large time. In summary, this periodicity allows to replace $a$ by its average along a circle, of radius of order $h^{-1}$:
\[ \langle \Op(a)u_h,u_h\rangle= \left\langle \Op\left(\frac{h\widehat{B}_0}{\pi}\int_0^{\frac{\pi}{h\widehat{B}_0}}a\circ\Phi_h^t\dd t\right)u_h,u_h\right\rangle.\]
 Then, we decompose $a$ according to its Fourier modes, and a stationnary phase argument shows that all non-zero modes yield negligible contributions. For technical reasons, it will turn out useful to first include the momentum cutoff $\chi_1$. Let us now proceed with the details of the proof.

\begin{proof} We start by including the cutoff $\chi_1$,
$$
\langle \Op(a)u_h,u_h\rangle= \langle \Op(a\chi_1)u_h,u_h\rangle+\langle \Op(a(1-\chi_1))u_h,u_h\rangle,
$$
and estimating the second term by standard microlocalization properties of the eigenfunctions. From the composition formula and~\eqref{e:symbol-laplacian}, we can write
\begin{multline*}
\langle \Op(a(1-\chi_1))u_h,u_h\rangle\\
=\langle \Op(a(1-\chi_1)(|\xi-h\alpha|^2-1)^{-1})(h^2\mathscr{L}_\alpha-1)u_h,u_h\rangle+\mathscr{O}_a(h),
\end{multline*}
where the constant in the remainder depends on a finite number of derivatives of $a$. Then, the eigenvalue equation $h^2 \mathscr{L}_\alpha u_h=u_h$ allows us to deduce that
\begin{equation}\label{e:cutoff-xi}
\langle \Op(a)u_h,u_h\rangle= \langle \Op(a\chi_1)u_h,u_h\rangle+\mathscr{O}_a(h).
\end{equation}

We now use the eigenvalue equation in the form \eqref{eq.commutator14} and \eqref{eq.average14}, and we get that, for every $a$ in $\mathscr{S}^0$,
$$
 \langle \Op(\chi_1\mathbf{X}_h(a))u_h,u_h\rangle=0,
$$
where we have used that $\mathbf{X}_h (\chi_1)=0$. We apply this formula to $a\circ\Phi_h^t$, where $\Phi_h^t$ is the flow induced by $\mathbf{X}_h$ and defined by~\eqref{e:magnetic-periodicflow}. We find that, for all $t\in\R$,
$$
 \langle \Op(\chi_1\mathbf{X}_h(a\circ \Phi_h^t))u_h,u_h\rangle=0.
$$
Integrating this expression, we get that
$$
\langle \Op(\chi_1a\circ \Phi_h^t)u_h,u_h\rangle=\langle \Op(\chi_1a)u_h,u_h\rangle.
$$
Then, we can integrate this invariance relation over a period of the flow and we obtain
\begin{equation}\label{e:cutoff-Egorov}
\langle \Op(a)u_h,u_h\rangle= \left\langle \Op\left(\chi_1\frac{h\widehat{B}_0}{\pi}\int_0^{\frac{\pi}{h\widehat{B}_0}}a\circ\Phi_h^t\dd t\right)u_h,u_h\right\rangle+\mathscr{O}_a(h).
\end{equation}
Recall the exact expression~\eqref{e:magnetic-periodicflow} for the flow $\Phi_h^t$, from which we deduce that
$$
\langle \Op(a)u_h,u_h\rangle= \left\langle \Op\left(\chi_1(\xi)\langle a\rangle_h(x,\xi)\right)u_h,u_h\right\rangle+\mathscr{O}_a(h),
$$
where the average $\langle a \rangle_h$ is defined by
\begin{multline*}
\langle a\rangle_h(x,\xi)\\
=\frac{1}{2\pi}\int_0^{2\pi}a\left(x+\frac{(\xi-h\alpha)^\perp}{h\widehat{B}_0}-J \frac{e^{tJ}}{h\widehat{B}_0}(\xi-h\alpha),e^{tJ}(\xi-h\alpha)+h\alpha\right) \dd t.
\end{multline*}
For $\xi-h\alpha\neq 0$, this can rewritten as an integral over the circle:
$$
\langle a\rangle_h(x,\xi)=\frac{1}{2\pi}\int_{\mathscr{S}^1}a\left(x+\frac{(\xi-h\alpha)^\perp}{h\widehat{B}_0}-J \frac{|\xi-h\alpha|\theta}{h\widehat{B}_0},|\xi-h\alpha|\theta+h\alpha\right) \dd\theta,
$$
which can be also decomposed using the Fourier decomposition of $a$ in the $x$-variable,
\begin{multline*}
\langle a\rangle_h(x,\xi)=\frac{1}{2\pi}\int_{\mathbf{S}^1}\widehat{a}_0\left(|\xi-h\alpha|\theta+h\alpha\right) d\theta
\\
+\frac{1}{2\pi}\sum_{k\neq 0}e^{2i\pi k\cdot\left(x+\frac{(\xi-h\alpha)^\perp}{h\widehat{B}_0}\right)}\int_{\mathbf{S}^1}e^{i\frac{2\pi|\xi-h\alpha| Jk}{h\widehat{B}_0}\cdot \theta }\widehat{a}_k\left(|\xi-h\alpha|\theta+h\alpha\right) \dd\theta.
\end{multline*}
We will now show that all non-zero modes have negligible contributions.

\begin{remark}
             Recall that there is a small abuse of notations by denoting by $d\theta$ the volume measure on the circle $\mathbf{S}^1=\{\theta\in\R^2:|\theta|=1\}$.
            \end{remark}
According to the Calder\'on-Vaillancourt Theorem~\ref{thm.cv}, each $k\neq 0$ yields the following contribution to the value of $\langle \Op(a)u_h,u_h\rangle$:
\begin{multline*}
\sum_{|\gamma|+|\beta|\leq N_0}h^{|\beta|}\langle k\rangle^{|\gamma|}\\
\times\left\|\partial_\xi^\beta\left(\chi_1(\xi)e^{2i\pi k\cdot\left(\frac{(\xi-h\alpha)^\perp}{h\widehat{B}_0}\right)}\int_{\mathbf{S}^1}e^{i\frac{2\pi|\xi-h\alpha| Jk}{h\widehat{B}_0}\cdot \theta }\widehat{a}_k\left(|\xi-h\alpha|\theta+h\alpha\right) \dd\theta\right)\right\|_\infty.
\end{multline*}
Recall now that the function $\chi_1(\xi)$ ensures that $|\xi-h\alpha|$ lies in the interval $[1/2,3/2]$ as it depends on $\xi$ through the variable $|\xi-h\alpha|$. Hence, the above sum can be bounded by
$$
C_1\langle k\rangle^{N_0}\sum_{|\beta|,\ell\leq N_0}
\sup_{\frac{1}{2}\leq |\xi+h\alpha |\leq \frac{3}{2}}\left|\int_{\mathbf{S}^1}e^{i\frac{2\pi|\xi-h\alpha| Jk}{h\widehat{B}_0}\cdot \theta }\left(\frac{k}{|k|}\cdot\theta\right)^\ell(\partial^{\beta}\widehat{a}_k)\left(|\xi-h\alpha|\theta+h\alpha\right) \dd\theta\right|,
$$
where $C_1>0$ depends only on $\chi$ and on $\widehat{B}_0$. In order to derive this bound from the previous one, note that, when differentiating $\widehat{a}_k$, this gives no rise to negative powers of $h$ while, when differentiating $e^{i\frac{2\pi|\xi-h\alpha| Jk}{h\widehat{B}_0}\cdot \theta }$, we get negative powers of $h$ which were compensated by the $h^{|\beta|}$ factor in front of the $L^\infty$ norm. This explains why there is no more $h^{|\beta|}$ factor in each term of the sum.

We are now in position to apply a standard stationary phase asymptotics~\cite[Th.~1, \S VIII.3.1]{Stein1993} to deduce that the oscillatory integral in this sum is bounded by
\begin{equation}
\left|\int_{\mathbf{S}^1}e^{i\frac{2\pi|\xi-h\alpha| Jk}{h\widehat{B}_0}\cdot \theta }\left(\frac{k}{|k|}\cdot\theta\right)^\ell(\partial^{\beta}\widehat{a}_k)\left(|\xi-h\alpha|\theta+h\alpha\right) \dd\theta\right|\leq C_2 h^{\frac{1}{2}} \|\widehat a_k\|_{\mathscr{C}^{N_2}},
\end{equation}
where $C_2,N_2>0$ do not depend on $a$ and $\alpha$ and where $C_2$ (but not $N_2$) depends also on $\widehat{B}_0$. Thanks to~\eqref{e:cutoff-Egorov}, we finally get that
\begin{multline*}
\langle \Op(a)u_h,u_h\rangle= \\ \left\langle \Op\left(\chi_1(\xi)\frac{1}{2\pi}\int_{\mathbf{S}^1}\widehat{a}_0\left(|\xi-h\alpha|\theta+h\alpha\right) \dd\theta\right)u_h,u_h\right\rangle+\mathscr{O}_a(h^{\frac{1}{2}}),
\end{multline*}
where the constant in the remainder depends on a finite number of derivatives of $a$ and on $\widehat{B}_0$. This concludes the proof of Lemma \ref{l:exact-egorov}.
\end{proof}

\subsection[Microlocalization]{Microlocalization near the unit cotangent bundle}\label{sec.microloc}
 We are now in position to prove Theorem~\ref{t:maintheo-quantitative}. To do so and according to Lemma~\ref{l:exact-egorov}, it is sufficient to analyze the following quantity:
$$
\left\langle\Op\left(\chi_1(\xi)\int_{S^*\T^2}a\left(x,|\xi-h\alpha|\theta+h\alpha\right)\dd x\dd\theta\right)u_h,u_h\right\rangle.
$$
We will show that $u_h$ is microlocalized where $|\xi| \sim |\xi -h\alpha| \sim 1$, and use this property to remove the last $\xi$-dependence. The above quantity can be split into two parts that we will analyze independently,
$$
A_1(h):=\left\langle\Op\left(\chi_1(\xi)\chi_{h^\frac12}(\xi)\int_{S^*\T^2}a\left(x,|\xi-h\alpha|\theta+h\alpha\right)\dd x\dd\theta\right)u_h,u_h\right\rangle,
$$
and
$$
A_2(h):=\left\langle\Op\left(\chi_1(\xi)\left(1-\chi_{h^\frac12}(\xi)\right)\int_{S^*\T^2}a\left(x,|\xi-h\alpha|\theta+h\alpha\right)\dd x\dd\theta\right)u_h,u_h\right\rangle.
$$
We begin with the case of $A_2(h)$. To that aim, we first observe that the symbol
$$
\tilde{a}_h(\xi):=\frac{{h^\frac12}(1-\chi_{h^\frac12}(\xi))}{|\xi-h\alpha|^2-1}\chi_1(\xi)\int_{S^*\T^2}a\left(x,|\xi-h\alpha|\theta+h\alpha\right)\dd x\dd\theta
$$
involves the function $\tilde{\chi}(t):=t^{-1}(1-\chi(t))$ which is a smooth function that is identically equal to $0$ on $[-1/4,1/4]$. Hence, even if the symbol has all its derivatives with respect to $\xi$ bounded, it is not uniformly in $\mathscr{S}^0$ as the derivatives are not bounded uniformly with respect to $h$. In fact, one has 
$$
\partial^\beta_\xi\left(\tilde{\chi}\left(\frac{|\xi-h\alpha|-1}{h^\frac12}\right)\right)=\mathscr{O}(h^{-\frac{|\beta|}{2}}).
$$
Yet, as $h^2\mathscr{L}_\alpha$, we can apply the last item of the composition formula to the pair $\tilde{a}_h$ and $|\xi-h\alpha|^2-1$ which gives an exact expression for the symbol of the operator for a fixed value of $h\in(0,1]$. Moreover, as all the involved symbols depend only on the $\xi$-variable, this gives the exact relation 
$$
A_2(h)=\left\langle\Op\left(\ldots\right)\frac{(h^2\mathscr{L}_\alpha-1)}{h^{\frac{1}{2}}}u_h,u_h\right\rangle+\mathscr{O}_a(h),
$$
where $\ldots$ is equal to
$$
\frac{\chi_1(\xi)}{|\xi-h\alpha|+1}\tilde{\chi}\left(\frac{|\xi-h\alpha|-1}{h^\frac12}\right)\int_{S^*\T^2}a\left(x,|\xi-h\alpha|\theta+h\alpha\right)\dd x\dd\theta.
$$
Using the eigenvalue equation for $u_h$, we find that $A_2(h)=\mathscr{O}_a(h),$ and we get that the proof of Theorem~\ref{t:maintheo-quantitative} reduces to estimating $A_1(h)$. To deal with this term, we apply the Taylor formula to write down
\begin{multline*}
 \int_{S^*\T^2}a\left(x,|\xi-h\alpha|\theta+h\alpha\right)\dd x\dd\theta=\int_{S^*\T^2}a\left(x,\theta\right)\dd x\dd\theta\\
 +\int_0^1\int_{S^*\T^2}(\partial_\xi a)\left(x,(1-t)\theta+t(|\xi-h\alpha|\theta+h\alpha)\right)\cdot((|\xi-h\alpha|-1)\theta+h\alpha)\dd t\dd x\dd\theta.
\end{multline*}
Again, the symbol 
\begin{multline*}
\chi_1(\xi)\chi_{h^\frac12}(\xi)\\
\times\int_0^1\int_{S^*\T^2}(\partial_\xi a)\left(x,(1-t)\theta+t(|\xi-h\alpha|\theta+h\alpha)\right)\cdot((|\xi-h\alpha|-1)\theta+h\alpha)\dd t\dd x\dd\theta
\end{multline*}
does not have its derivative uniformly bounded with respect to $h$. It only verifies that the derivatives of order $\beta$ are of size $\mathscr{O}(h^{\frac{1-|\beta|}{2}})$. Using the Calder\'on-Vaillancourt Theorem~\ref{thm.cv}, this implies that
$$
A_1(h)=\int_{S^*\T^2}a\left(x,\theta\right)\dd x\dd\theta\left\langle\Op\left(\chi_1(\xi)\chi_{h^\frac12}(\xi)\right)u_h,u_h\right\rangle+\mathscr{O}_a(h^{\frac12}).
$$
Applying our proof with $a\equiv 1$ implies that 
$$
\left\langle\Op\left(\chi_1(\xi)\chi_{h^\frac12}(\xi)\right)u_h,u_h\right\rangle=1+\mathscr{O}(h^{\frac12}),
$$
which concludes the proof of Theorem~\ref{t:maintheo-quantitative}.

\section{Magnetic semiclassical measures}
\label{s:semiclassicalmeasure}

In the general case, when the magnetic field $B$ is not constant, there is no periodic classical flow that we can use to apply the averaging method. Therefore, we need to use other techniques in order to catch the subleading effects of the magnetic field, and prove Theorem \ref{t:maintheo}. We start this section by recalling standard properties of semiclassical measures on the torus.

 In the following, we will say that a probability measure $\nu$ on $\T^2$ is a \emph{quantum limit} if there exists a sequence $(u_{h_n})_{n\geq 1}$ in $H^2(\T^2,L)$ solving the approximate eigenvalue equation
\begin{equation}\label{e:semiclassical-quasimode}
 h_n^2\mathscr{L}_\alpha u_{h_n}=u_{h_n}+o_{L^2}\Big(h_n^{\frac{3}{2}}\Big),\quad \|u_{h_n}\|_{L^2(\T^2)}=1,\quad h_n\rightarrow 0^+,
\end{equation}
and such that
\begin{equation}\label{e:defQL}
 \forall a\in\mathscr{C}^0(\T^2,\C),\quad\lim_{n\rightarrow \infty}\int_{\T^2}a(x) |u_{h_n}(x)|^2\dd x=\int_{\T^2}a(x)\nu(\dd x).
\end{equation}
Note that, as $u_{h_n}$ belongs to $L^2(\T^2,L)$, the function $|u_{h_n}(x)|^2$ belongs to $L^1(\T^2,\R_+)$. Moreover, as $\T^2$ is a compact metric space, the set of probability measures on $\T^2$ is compact (for the weak-$\star$ topology) and we can always extract converging subsequences from a given sequence of probability measures. The set of quantum limits for $\mathscr{L}_\alpha$ will be denoted by $\mathcal{N}(\mathscr{L}_\alpha).$ Our goal is to prove that the only possible quantum limit is the Lebesgue measure, i.e. $\mathcal N( \mathscr{L}_\alpha) = \lbrace {\rm{Leb}} \rbrace$. 

Let us start by recalling how these quantum limits can be lifted to the phase space $T^*\T^2$ using the magnetic Weyl quantization and the notion of semiclassical measures. We will also describe some elementary properties of these lifted measures.

\begin{remark}
As usual with semiclassical methods, we will drop the index $n$ in the following and rather write $(u_h)_{h\rightarrow 0^+}$. We will also refer to such families as sequences though the notation may be misleading.
\end{remark}

\subsection[Lifing quantum limits]{Lifting quantum limits to the cotangent bundle}

 In order to study quantum limits, it is standard to lift these measures to distributions on $T^*\T^2 = \T^2 \times \R^2$ through the following definition 
$$
w_h:a\in\mathscr{C}_c^{\infty}(T^*\T^2,\C)\mapsto \left\langle \Op(a)u_h,u_h\right\rangle_{L^2(\T^2,L)},
$$
where we recall that the scalar product $\langle u, v \rangle = \int_{\T^2}u(x)\overline{v}(x)dx$ for $u,v \in L^2(\T^2,L)$ is independent of the choice of the fundamental domain for $\T^2$. We list a few properties of these distributions in the following Lemma.
\begin{lemma}\label{l:magnetic-wigner} Let $(u_h)_{h\rightarrow 0^+}$ be a sequence verifying\footnote{In fact, for this lemma to be true, only a $o_{L^2}(h)$ is required.}~\eqref{e:semiclassical-quasimode}. Then the following holds:
\begin{enumerate}
 \item The corresponding sequence $(w_h)_{h\rightarrow 0^+}$ is bounded in $\mathscr{D}'(T^*\T^2)$. In particular, up to extraction, the sequence $(w_h)_{h\rightarrow 0^+}$ converges in $\mathscr{D}'(T^*\T^2)$ to some $\mu$ as $h\rightarrow 0^+$.
 \item Any accumulation point $\mu$ of $(w_h)_{h\rightarrow 0^+}$ is a probability measure on 
$$
S^*\T^2:=\{(x,\xi)\in T^*\T^2:\ |\xi |=1\},
$$
which is invariant by the geodesic flow 
$$\varphi^t:(x,\xi)\in T^*\T^2\mapsto (x+t\xi,\xi)\in T^*\T^2.$$
\item The measure $\nu = \pi_*\mu$ is a quantum limit for $(u_h)_{h \to 0^+}$, i.e. $\nu$ satisfies \eqref{e:defQL}, where $\pi:(x,\xi)\in T^*\T^2\mapsto x\in\T^2$.
\end{enumerate} 
\end{lemma}

In the following, these probability measures will be refered as (magnetic) semiclassical measures and we will denote their set by $\mathcal{M}(\mathscr{L}_\alpha).$

\begin{proof} The proof of these facts is now standard in the case of Laplace eigenfunctions~\cite[Ch.~5]{Zworski} and we briefly explain how it can be directly translated in our magnetic setting. The main argument is that the principal symbol of $h^2 \mathscr{L}_\alpha$ is $|\xi|^2$, from which we deduce microlocalization properties of the eigenfunctions. For the first point, this follows directly from the fact that
$$
\|\Op(a)\|_{L^2(\T^2,L)\rightarrow L^2(\T^2,L)}\leq \frac{1}{(2\pi)^4}\sum_{\xi\in 2\pi\Z^2}\int_{\R^2}|\mathcal{F}(a)(\eta,\zeta)|\dd\zeta\leq C\|a\|_{\mathscr{C}^8},
$$
where the constant depends only on the support of $a$. Here we used~\eqref{eq.Opa} but we could as well have used the Calder\'on-Vaillancourt Theorem.

For the second point, we fix some smooth function $a\geq 0$ compactly supported in $T^*\T^2$ and some $\varepsilon>0$. We set $a_{\varepsilon}:=\sqrt{a+\varepsilon}\in \mathscr{S}^0$. Then, one can use the composition rule in Theorem $\ref{thm.comp}$ together with the formula \eqref{eq.adjoint} for the adjoint and the Calder\'on-Vaillancourt Theorem \ref{thm.cv} to write that 
$$
\|\Op(a_\varepsilon)u_h\|_{L^2}^2=\langle\Op(a_\varepsilon)^*\Op(a_\varepsilon)u_h,u_h\rangle=\langle \Op((a+\varepsilon))u_h,u_h\rangle+\mathscr{O}_\varepsilon(h).
$$
Hence, for every $a\in\mathscr{C}^\infty_c(T^*\T^2,\R_+)$ and for every $\varepsilon>0$, $\langle \mu,a\rangle\geq -\varepsilon $ from which we can infer that $\mu$ is a positive distribution, hence a measure. From the fact that $u_h$ is normalized in $L^2$, we also know that it has total mass $\leq 1$. 

To see that it has total mass equal to $1$, we fix a smooth function $a$ on $T^*\T^2$ which is identically equal to $1$ for $|\xi|\leq 2$ and identically $0$ for $|\xi|\geq 4$. We also fix $R>1$ and we set $a_R(x,\xi)=a(x,\xi/R)$. From the composition rule in Theorem \ref{thm.comp} and the Calder\'on-Vaillancourt Theorem \ref{thm.cv}, one finds
\begin{align*}
1 &= \langle\Op(a_R) u_h,u_h\rangle+ \langle \langle\Op(1-a_R) u_h,u_h\rangle \\
&=\langle\Op(a_R) u_h,u_h\rangle+\left\langle \Op(b_R)\left(\Op(|\xi|^2)-\text{Id}\right)u_h,u_h\right\rangle +\mathscr{O}(h),
\end{align*}
where $b_R(x,\xi)= (1-a_R(x,\xi))/(|\xi|^2-1)$ belongs to $\mathscr{S}^{-2}(T^*\T^2)$. Recall now from~\eqref{e:symbol-laplacian} that 
\begin{multline}\label{e:quantization-magnetic-laplacian}
h^2\mathscr{L}_\alpha
=\Op\left((\xi_1-h\alpha_1-hA_1^{\text{per}})^2+(\xi_2-h\alpha_2-hA_2^{\text{per}})^2\right)\\
= \Op(|\xi|^2+ h c_1(x)\cdot\xi+h^2c_0(x)),
\end{multline} where $c_0$ and $c_1$ belong to $\mathscr{C}^{\infty}(\T^2,\R)$ and to $\mathscr{C}^{\infty}(\T^2,\R^2)$ respectively. Hence, applying the composition rule one more time together with the Calder\'on-Vaillancourt Theorem, one finds that
$$
1= \langle\Op(a_R) u_h,u_h\rangle+\left\langle \Op(b_R)\left(h^2\mathscr{L}_\alpha-\text{Id}\right)u_h,u_h\right\rangle +\mathscr{O}(h).
$$
Using the eigenmode equation~\eqref{e:semiclassical-quasimode}, we find (after letting $h\rightarrow 0^+$) that $\langle\mu,a_R\rangle= 1$. Applying the dominated convergence Theorem, we find that $\mu$ is a probability measure. Along the way, we can verify with the same argument that $\mu$ is supported on $S^*\T^2$. Indeed, taking a function $a$ that is supported outside $\{|\xi|=1\}$, we find with the exact same argument that $\langle\mu,a\rangle =0$. Hence, $\text{supp}(\mu)\subset\{|\xi|=1\}$.

Recall that so far we only used that the error in the quasimode equation~\eqref{e:semiclassical-quasimode} is $o(1)$ as $h\rightarrow 0^+$. We are left with proving the invariance by the geodesic flow that will require the remainder to be $o(h)$. To see this, we use the quasimode equation to write 
$$
\left\langle\left[\Op(a),h^2\mathscr{L}_\alpha\right]u_h,u_h\right\rangle=o(h).
$$
Using the composition rule for pseudodifferential operators one last time, we find that
$$
\left\langle\Op(\{|\xi|^2,a\})u_h,u_h\right\rangle=o(1),
$$
from which we deduce that $\int_{S^*\T^2}\xi\cdot\partial_xa(x,\xi)\mu(\dd x,\dd\xi)=0$. This implies the invariance by the geodesic flow. Finally, the last item is obtained by using symbols $a(x)$ independent of $\xi$, by considering test functions of the form $a_R(x,\xi)=a(x)\chi(\|\xi\|^2/R)$ (with $\chi$ compactly supported and equal to $1$ near $0$) and by letting $R\rightarrow \infty$ (after letting $h\rightarrow 0^+$).
\end{proof}

\subsection{Decomposition of invariant measures}\label{s:decomposition}

 The quantum limits $\nu$ we are interested in are the pushforward of measures $\mu$ on $S^*\T^2$, that are invariant by the geodesic flow. In order to analyze the regularity of $\nu$, we can study the regularity of $\mu$. To that aim, we follow the strategy of~\cite{Macia2010, AnantharamanMacia2014}. Indeed, any geodesic-invariant measure can be decomposed according to the periodic orbits of the geodesic flow.
 
Periodic geodesics can be classified as follows. We denote by $\mathcal{L}_1$ the set of primitive sublattices $\Lambda$ of $\Z^2$ that are of rank $1$. This means that $\text{dim}\langle\Lambda\rangle=1$ and that $\langle \Lambda\rangle\cap\Z^2=\Lambda$ where $\langle\Lambda\rangle$ is the $\R$-vector space generated by $\Lambda$. For every $\Lambda\in\mathcal{L}_1$, we fix $\mathfrak{e}_\Lambda$ such that $\Z\mathfrak{e}_\Lambda=\Lambda$ and we denote by $\mathfrak{e}_\Lambda^\perp$ the vector that is directly orthogonal to it, with same length. We set $L_\Lambda:=|\mathfrak{e}_\Lambda |$ and $\Lambda^\perp:=\R \mathfrak{e}_\Lambda^\perp$. 

On the one hand, the geodesic flow on $T^*\T^2$ is periodic in each direction $\Lambda^{\perp}$, i.e. for every $a\in\mathscr{C}^\infty_c(T^*\T^2)$, one has
$$\forall (x,\xi)\in\T^2\times(\Lambda^\perp\setminus\{0\}),\quad  \lim_{T\rightarrow+\infty} 
\frac{1}{T}\int_0^Ta(x+t\xi,\xi)dt=\mathcal{I}_\Lambda(a)(x,\xi),$$
where
$$
\mathcal{I}_\Lambda(a)(x,\xi):=\sum_{k\in\Lambda}\widehat{a}_k(\xi)e^{2i\pi k\cdot x}.
$$
On the other hand, for every $\xi$ in
$$
\Omega_2:=\R^2\setminus\left(\bigcup_{\Lambda\in\mathscr{L}_1}\Lambda^\perp\right),
$$
one has, for every $a\in\mathscr{C}^\infty_c(T^*\T^2)$,
$$
\forall x\in\T^2,\quad \lim_{T\rightarrow+\infty} 
\frac{1}{T}\int_0^Ta(x+t\xi,\xi)\dd t=\int_{\T^2}a(y,\xi)\dd y.
$$
\begin{remark}\label{rem.Binfty}
 When $a$ is independent of $\xi$, we observe that the function $\mathcal{I}_\Lambda(a)$ is also independent of $\xi$. Note that, for $a(x,\xi)=B(x)$ and for $\theta$ belonging to $\Lambda^\perp(\theta)\cap\mathbf{S}^1$, one recovers the function $B_\infty$ from the introduction as $\mathcal{I}_{\Lambda(\theta)}(B)(x)=B_\infty(x,\theta)$ in that case.
\end{remark}

Thanks to these observations, we can now decompose each invariant probability measure $\mu$ on $S^*\T^2$ as follows:
$$
\mu=\mu|_{\T^2\times(\Omega_2\cap\mathbf{S}^1)}+\sum_{\Lambda\in\mathcal{L}_1}\mu|_{\T^2\times(\Lambda^\perp\cap\mathbf{S}^1)},
$$
where each term in this sum is a nonnegative and finite measure that is invariant by the geodesic flow. We can also write the Fourier decomposition of $\mu$, i.e.
$$
\mu(x,\xi):=\sum_{k\in\Z^2}\widehat{\mu}_k(\xi)e^{2i\pi x\cdot k},\quad\text{and}\quad\forall\Lambda\in\mathcal{L}_1,\ \mathcal{I}_\Lambda(\mu)=\sum_{k\in\Lambda}\widehat{\mu}_k(\xi)e^{2i\pi x\cdot k}.
$$
We then have the following general result~\cite[\S2]{AnantharamanMacia2014}.
\begin{lemma}\label{l:anantharamanmacia} Let $\mu$ be a probability measure on $S^*\T^2$ that is invariant by the geodesic flow. Then, the following holds:
\begin{enumerate}
 \item For any $\Lambda$ in $\mathcal{L}_1$, $\mathcal{I}_\Lambda(\mu)$ is a finite nonnegative Radon measure on $S^*\T^2$.
 \item For any $\Lambda$ in $\mathcal{L}_1$, we have $\mu|_{\T^2\times(\Lambda^\perp\cap\mathbf{S}^1)}=\mathcal{I}_\Lambda(\mu)|_{\T^2\times(\Lambda^\perp\cap\mathbf{S}^1)}$.
 \item $\widehat{\mu}_0$ is a finite nonnegative Radon measure on $S^*\T^2$ and $\mu|_{\T^2\times(\Omega_2\cap\mathbf{S}^1)}=\widehat{\mu}_0|_{\T^2\times(\Omega_2\cap\mathbf{S}^1)}$.
\end{enumerate}
\end{lemma}
Note that the proof of this Lemma is purely dynamical: it does not require $\mu$ to be a semiclassical measure. The fact that our semiclassical measures $\mu$ are indeed invariant by the geodesic flow allows us to decompose them accordingly,
\begin{equation}\label{eq.decomposition14}
\mu=\widehat{\mu}_0|_{\T^2\times(\Omega_2\cap\mathbf{S}^1)}+\sum_{\Lambda\in\mathcal{L}_1}\mathcal{I}_\Lambda(\mu)|_{\T^2\times(\Lambda^\perp\cap\mathbf{S}^1)}.
\end{equation}
Hence, in order to understand the regularity of $\mu$ along the $x$-variable (and thus of its pushforward $\nu$ on $\T^2$), we need to understand the regularity of each individual measure $\mathcal{I}_\Lambda(\mu)|_{\T^2\times(\Lambda^\perp\cap\mathbf{S}^1)}.
$ This is the content of the next section.

\section[Periodic orbits]{Analyzing the semiclassical measure along periodic orbits}\label{s:periodic}

We consider a measure $\mu$ which is an accumulation point of the distribution $w_h$, with 
\[ w_h(a) = \langle \Op(a) u_h, u_h \rangle, \qquad a \in \mathscr{S}^0(\T^2 \times \R^2).\]
We have seen with Lemma \ref{l:magnetic-wigner} that $\mu$ must be invariant by the geodesic flow, and therefore it can be decomposed as in \eqref{eq.decomposition14}. In this section we analyse the contribution of periodic orbits, i.e. we study the measure $\mathcal{I}_\Lambda(\mu)$, for any $\Lambda \in \mathcal L_1$.

We will split the analysis into two parts, depending on whether the direction $\xi$ is close to $\Lambda^\perp$ or not. To that aim, we fix a sequence $R_m\rightarrow +\infty$ and we let $\chi$ be a smooth function on $\R$ which is identically equal to $1$ (resp. $0$) on $[-1,1]$ (resp. outside $[-2,2]$). We define
$$
\chi_{R_mh^{\frac{1}{2}},\Lambda}(\xi):=\chi\left(\frac{\xi\cdot\mathfrak{e}_{\Lambda}}{R_mh^{\frac{1}{2}}L_\Lambda}\right),
$$
where we recall that $L_\Lambda = \| \mathfrak{e}_{\Lambda} \|$, and we split $w_h$ accordingly:
$$
\forall a\in\mathscr{C}^\infty_c(T^*\T^2),\quad\langle w_h,a\rangle=\left\langle w_h,a\chi_{R_mh^{\frac{1}{2}},\Lambda}\right\rangle+\left\langle w_h,a\left(1-\chi_{R_mh^{\frac{1}{2}},\Lambda}\right)\right\rangle.
$$
Again, we will drop the index for the sequence $(R_m)_{m\geq 1}$ and just write $R\rightarrow \infty$. This leads to the following definitions:
$$
w_{h,R,\Lambda}: a\in\mathscr{C}^\infty_c(T^*\T^2)\mapsto \left\langle w_h,a\chi_{Rh^{\frac{1}{2}},\Lambda}\right\rangle,
$$
and
$$w_{h,R}^{\Lambda}: a\in\mathscr{C}^\infty_c(T^*\T^2)\mapsto \left\langle w_h,a\left(1-\chi_{Rh^{\frac{1}{2}},\Lambda}\right)\right\rangle.
$$
We will analyze the possible accumulation points of each of these sequences independently, as $h \to 0$, and for $R$ very large. In fact, we will use the notation 
$$(w_{h,R,\Lambda})_{h\to 0^+,R\to +\infty},$$ 
where we implicitly mean that we take the limit $h \to 0^+$ and then $R \to \infty$, in this precise order. Then, we have the following analogue of Lemma~\ref{l:magnetic-wigner} for each of these sequences.

\begin{lemma}\label{l:magnetic-wigner-split} Let $(u_h)_{h\rightarrow 0^+}$ be a sequence verifying~\eqref{e:semiclassical-quasimode} with a unique\footnote{Up to extraction, this is always the case.} semiclassical measure $\mu$. Then the following holds:
\begin{enumerate}
 \item The sequences $(w_{h,R,\Lambda})_{h\rightarrow 0^+,R\rightarrow +\infty}$ and $(w_{h,R}^{\Lambda})_{h\rightarrow 0^+, R\rightarrow +\infty}$ are bounded in $\mathscr{D}'(T^*\T^2)$;
 \item Any accumulation point $\mu_\Lambda$ and $\mu^{\Lambda}$ of the sequences $(w_{h,R,\Lambda})_{h\rightarrow 0^+,R\rightarrow +\infty}$ and $(w_{h,R}^{\Lambda})_{h\rightarrow 0^+, R\rightarrow +\infty}$ are finite nonnegative measure on $S^*\T^2$
which are invariant by the geodesic flow;
\item $\mu=\mu^{\Lambda}+\mu_{\Lambda}$.
\end{enumerate} 
\end{lemma}

Let us draw a few consequences of this Lemma. First, one has
\begin{equation}\label{e:splitting-2microlocal-measure}
\mu|_{\T^2\times(\Lambda^\perp\cap\mathbf{S}^1)}=\mu^{\Lambda}|_{\T^2\times(\Lambda^\perp\cap\mathbf{S}^1)}+\mu_{\Lambda}|_{\T^2\times(\Lambda^\perp\cap\mathbf{S}^1)}.
\end{equation}
Also, since $\chi_{R h^{\frac 12}, \Lambda}$ is supported on an $h^{\frac 12}$-neighborhood of $\Lambda^\perp$, it follows by construction that $\mu_\Lambda$ is supported on $\Lambda^\perp$ i.e. $\mu_{\Lambda}|_{\T^2\times(\Lambda^\perp\cap\mathbf{S}^1)}=\mu_\Lambda$. Thanks to the invariance by the geodesic flow, we can apply Lemma~\ref{l:anantharamanmacia} to the measures $\mu$, $\mu^{\Lambda}$ and $\mu_{\Lambda}$ and we deduce
\begin{equation}\label{e:splitting-2microlocal-measure2}
\mu|_{\T^2\times(\Lambda^\perp\cap\mathbf{S}^1)}=\mathcal{I}_\Lambda(\mu)|_{\T^2\times(\Lambda^\perp\cap\mathbf{S}^1)}=\mathcal{I}_\Lambda(\mu^{\Lambda})|_{\T^2\times(\Lambda^\perp\cap\mathbf{S}^1)}+\mathcal{I}_\Lambda(\mu_{\Lambda}).
\end{equation}
In particular, in view of dealing with the regularity of $\mu|_{\T^2\times(\Lambda^\perp\cap\mathbf{S}^1)}$, we only need to deal with the regularity of $\mu_\Lambda$ and $\mu^{\Lambda}$ and to work with test functions having all their Fourier coefficients $k$ in $\Lambda$.

\begin{proof}[Proof of Lemma \ref{l:magnetic-wigner-split}] Using~\eqref{e:change-semiclassical-parameter}, we can rescale the quantization and write
 $$
 \langle w_{h,R,\Lambda},a\rangle=\left\langle \mathsf{Op}_{h^{\frac{1}{2}}}^B\Big(a\big(x,h^{\frac{1}{2}}\xi\big)\chi_{R,\Lambda}(\xi)\Big)u_h,u_h\right\rangle,
 $$
 and
 $$
 \langle w_{h,R}^{\Lambda},a\rangle=\left\langle \mathsf{Op}_{h^{\frac{1}{2}}}^B\Big(a\big(x,h^{\frac{1}{2}}\xi\big)(1-\chi_{R,\Lambda}(\xi))\Big)u_h,u_h\right\rangle.
 $$
Without loss of generality, we can assume $0<h<1$ and $R>1$ so that all the derivatives of the symbol $a\big(x,h^{\frac{1}{2}}\xi\big)\chi_{R,\Lambda}(\xi)$ are uniformly bounded in terms of $h$ and $R$, and the involved symbols belong to the class $\mathscr{S}^0$. Hence, arguing as in the proof of item 1 of Lemma~\ref{l:magnetic-wigner-split}, we get the first and the third items of the present Lemma and the fact that the accumulation points are nonnegative measures. As $\mu$ is supported on $S^*\T^2$, we deduce that both $\mu_\Lambda$ and $\mu^{\Lambda}$ are also supported on $S^*\T^2$. Hence, the only thing to check is the invariance by the geodesic flow. To see this, we repeat the same argument as in the proof of Lemma~\ref{l:magnetic-wigner} with the semiclassical parameter $h^{\frac{1}{2}}$ instead of $h$. Using the eigenvalue equation \eqref{e:semiclassical-quasimode} and dividing by $h$ we have
\begin{equation}\label{eq.name0}
 \left\langle \left[\mathsf{Op}_{h^{\frac{1}{2}}}^B\left(a\big(x,h^{\frac{1}{2}}\xi\big)\tilde{\chi}_{R,\Lambda}(\xi)\right),h\mathscr{L}_\alpha \right]u_h,u_h\right\rangle=o(h^{\frac{1}{2}}),
\end{equation}
for both $\tilde{\chi}_{R,\Lambda}=\chi_{R,\Lambda}$ or $1-\chi_{R,\Lambda}$. Recall that $h\mathscr{L}_{\alpha}=\mathsf{Op}_{h^{\frac{1}{2}}}^B\big(|\xi |^2+ h^{\frac{1}{2}} c_1(x)\cdot\xi+hc_0(x)\big)$, where $c_0$ and $c_1$ belong to $\mathscr{C}^{\infty}(\T^2,\C)$ and to $\mathscr{C}^{\infty}(\T^2,\C^2)$ respectively. Applying the composition rule from Theorem \ref{thm.comp} together with the fact that $|\xi |^2$ is quadratic of order $2$ (and independent of $x$) and the support properties of $a$, $\tilde{\chi}_{R,\Lambda}$ and $\nabla\tilde{\chi}_{R,\Lambda}$, we claim that
\begin{multline}\label{eq.name}
 \left\langle \left[\mathsf{Op}_{h^{\frac{1}{2}}}^B\left(a\big(x,h^{\frac{1}{2}}\xi\big)\tilde{\chi}_{R,\Lambda}(\xi)\right),\mathsf{Op}_{h^{\frac{1}{2}}}^B\left(|\xi |^2\right) \right]u_h,u_h\right\rangle\\
 = 2i h^{\frac{1}{2}}\left\langle \mathsf{Op}_{h^{\frac{1}{2}}}^B\left(\xi\cdot\partial_x a\big(x,h^{\frac{1}{2}}\xi\big)\tilde{\chi}_{R,\Lambda}(\xi)\right)u_h,u_h\right\rangle +o(h^{\frac{1}{2}}),
\end{multline}
where the remainder term may depend on $R$. Similarly, we also claim that the remaining term satisfies
\begin{equation}\label{eq.name2}
\left\langle \left[\mathsf{Op}_{h^{\frac{1}{2}}}^B\left(a\big(x,h^{\frac{1}{2}}\xi\big)\tilde{\chi}_{R,\Lambda}(\xi)\right),\mathsf{Op}_{h^{\frac{1}{2}}}^B\left(h^{\frac{1}{2}} c_1(x)\xi+hc_0(x)\right) \right]u_h,u_h\right\rangle=o(h^{\frac{1}{2}}).
\end{equation}
Note that, since the symbols involved are not uniformly in $\mathscr{S}^0$, equations \eqref{eq.name} and \eqref{eq.name2} are not straightforward. We provide some details in the case of \eqref{eq.name2}, the other one being similar. First note that the
$h^{1/2}a\big(x,h^{\frac{1}{2}}\xi\big)\tilde{\chi}_{R,\Lambda}(\xi)$ is uniformly in $\mathscr{S}^{-1}$. Indeed, the support properties of $a$ ensures that $h^{\frac12}\xi$ is uniformly bounded so that $h^{\frac12}=\mathscr{O}(\langle\xi\rangle^{-1})$ and differentiating with respect to $(x,\xi)$ does not give any growth in $\langle\xi\rangle$ or $h$. Therefore, since $c_1(x)  \cdot\xi \in \mathscr{S}^1$, the composition Theorem \ref{thm.comp} gives
\begin{multline*}
\left\langle \left[\mathsf{Op}_{h^{\frac{1}{2}}}^B\left(a\big(x,h^{\frac{1}{2}}\xi\big)\tilde{\chi}_{R,\Lambda}(\xi)\right),\mathsf{Op}_{h^{\frac{1}{2}}}^B\left(h^{\frac{1}{2}} c_1(x)\cdot\xi \right) \right]u_h,u_h\right\rangle \\= \frac{ h^{\frac 12}}{i} \left\langle \mathsf{Op}_{h^{\frac{1}{2}}}^B \big\lbrace h^{\frac 12} a \tilde{\chi}_{R,\Lambda}, c_1(x)\cdot \xi \big\rbrace  u_h, u_h \right\rangle + o(h^{\frac{1}{2}}).
\end{multline*}
Using the support properties of $a$ and $\tilde{\chi}_{R,\Lambda}$, we claim that
\[  \big\lbrace  a(x,h^{\frac 12} \xi) \tilde{\chi}_{R,\Lambda}(\xi), c_1(x)\cdot \xi \big\rbrace \in \mathscr{S}^0, \]
uniformly. This last observation follows from the fact that the only growth in $\xi$ may come from the fact that $c_1(x)\cdot \xi$ is differentiated with respect to $x$. Yet, in such a case,  either one differentiates $a$ with respect to $\xi$ and a factor $h^{1/2}$ comes out (recall that $h^{1/2}\xi$ is bounded thanks to the support properties of $a$), or one differentiates $\tilde{\chi}_{R,\Lambda}$ (and $\xi$ becomes bounded). Thus,
\[\left\langle \left[\mathsf{Op}_{h^{\frac{1}{2}}}^B\left(a\big(x,h^{\frac{1}{2}}\xi\big)\tilde{\chi}_{R,\Lambda}(\xi)\right),\mathsf{Op}_{h^{\frac{1}{2}}}^B\left(h^{\frac{1}{2}} c_1(x)\xi \right) \right]u_h,u_h\right\rangle = o(h^{\frac 12}). \]
The term $c_0(x)$ in \eqref{eq.name2} is easily estimated since it is in $\mathscr{S}^0$. 

Now, combining \eqref{eq.name0}, \eqref{eq.name} and \eqref{eq.name2}, and dividing by $h^{\frac 12}$ we find
\[ \lim_{h \to 0} \left\langle \mathsf{Op}_{h^{\frac{1}{2}}}^B\left(\xi\cdot\partial_x a\big(x,h^{\frac{1}{2}}\xi\big)\tilde{\chi}_{R,\Lambda}(\xi)\right)u_h,u_h\right\rangle = 0. \]
We deduce that the measures $\mu_\Lambda$ and $\mu^\Lambda$ are invariant by the geodesic flow.

\end{proof}

\subsection{Analyzing the noncompact part}\label{sec.noncompactpart}

 We start by analyzing the ``noncompact'' part of the measure, meaning $\mu^{\Lambda}$.
\begin{lemma}\label{lem.muup}
Let $k\in\Lambda\setminus \{0\}$ and let $a\in\mathscr{C}^\infty_c(\R^2)$. Then, one has
$$
\int_{S^*\T^2} e^{2i\pi k\cdot x} a(\xi)\mu^{\Lambda}(\dd x,\dd\xi)=0.
$$
In particular, the nonnegative Radon measure $\mathcal{I}_\Lambda(\mu^{\Lambda})(x,\xi)=\int_{\T^2}\mu^{\Lambda}(\dd x,\xi)$ is independent of $x$.
\end{lemma}

As a direct corollary of this Lemma and of~\eqref{e:splitting-2microlocal-measure} and~\eqref{e:splitting-2microlocal-measure2}, one has
\begin{equation}\label{e:splitting-2microlocal-measure3}
 \mu|_{\T^2\times(\Lambda^{\perp}\setminus\{0\})}=\mathcal{I}_\Lambda(\mu_\Lambda)+\int_{\T^2}\mu^{\Lambda}(\dd x,\xi)|_{\T^2\times(\Lambda^{\perp}\setminus\{0\})}.
\end{equation}
In particular, in view of proving that $\mu|_{\T^2\times(\Lambda^{\perp}\setminus\{0\})}$ is independent of $x$, it only remains to analyze the properties of $\mathcal{I}_\Lambda(\mu_\Lambda)$ which will be the content of Lemmas~\ref{l:magnetic-wigner-2microlocal} and~\ref{l:invariance-2microlocal} below.

The first idea in order to prove Lemma \ref{lem.muup} will be to use the eigenvalue equation for $u_h$ in the form
\begin{equation}\label{eq.1110}
 \left\langle \left[ \Op \left(b (1-\chi_{Rh^{1/2},\Lambda})\right) , h^2 \mathscr{L}_\alpha \right] u_h,u_h \right\rangle = o(h^{\frac 32}),
 \end{equation}
with a well-chosen symbol $b$. One will notice that the cutoff does not belong to a nice class of symbols uniformly with respect to $h$. In fact, we will need this $h^{\frac 12}$ scaling in the later analysis of Section \ref{sec.compactpart}. However this power is critical in the sense that, when $R$ is very large, the composition rule for pseudodifferential operators will still apply in this case, so that \eqref{eq.1110} will give
\[\lim_{R \to \infty} \lim_{h\to 0} \left\langle \Op \left( \left\{ b(1-\chi), |\xi|^2 \right\} \right) u_h, u_h \right\rangle =0 .\]
We will chose $b$ such that $\lbrace b, |\xi|^2 \rbrace = e^{2i\pi k x} a(\xi)$, in which case we get the desired result
\[ \langle \mu^\Lambda,  e^{2i\pi k x} a(\xi) \rangle = 0. \]
Note that, for this argument to work we only use that the magnetic field is small enough (i.e. a subprincipal effect on the symbol). Also, in order to deal with singular symbols involving the cutoff function $\chi$, and to show that the composition rule applies, we find it useful to work with the rescaled quantization $\mathsf{Op}_{h^{\frac 12}}^B$.

\begin{proof}
 Let $k\in\Lambda\setminus \{0\}$ and $a\in\mathscr{C}^\infty_c(\R^2)$. We want to analyze the limit of $\langle w_{h,R}^{\Lambda},a(\xi)e^{2i\pi k\cdot x}\rangle$ as $h\rightarrow 0^+$ and $R\rightarrow \infty$ (in this order). To do this, we introduce the following symbol
\begin{equation}
s_{h,R}(x,\xi) = e^{2i\pi k\cdot x}\frac{a(h^{\frac 12}\xi)}{\xi\cdot\mathfrak{e}_\Lambda}\left(1-\chi\left(\frac{\xi\cdot\mathfrak{e}_\Lambda}{L_\Lambda R}\right)\right),
\end{equation}
which belongs to the class $\mathscr{S}^0$ uniformly with respect to $h$. Indeed, $\xi\cdot\mathfrak{e}_\Lambda$ is bounded from below by a constant so that the function is bounded uniformly and any derivative with respect to $(x,\xi)$ is bounded for the same reason. From the eigenvalue equation \eqref{e:semiclassical-quasimode} and the Calder\'on-Vaillancourt Theorem~\ref{thm.cv}, we know that
\begin{equation}
\left\langle \left[\mathsf{Op}_{h^{\frac 12}}^B \left(s_{h,R} \right),h^2\mathscr{L}_\alpha\right]u_h,u_h\right\rangle =o\left(h^{\frac32}\right).
\end{equation}
We write $\mathscr{L}_\alpha$ in $h^{1/2}$-quantization,
\begin{align*}
h^2\mathscr{L}_\alpha& = \Op(|\xi |^2+ h c_1(x)\cdot\xi+h^2c_0(x)) \\
&= h \mathsf{Op}_{h^{\frac 12}}^B( | \xi |^2 + h^{\frac{1}{2}} c_1(x) \cdot \xi + h c_0(x)),
\end{align*}
and divide by $h$ to find
\begin{equation}\label{eq.testup}
\left\langle \left[\mathsf{Op}_{h^{\frac 12}}^B \left(s_{h,R} \right),\mathsf{Op}_{h^{\frac 12}}^B(| \xi |^2 + h^{\frac{1}{2}} c_1(x) \cdot \xi + h c_0(x))\right]u_h,u_h\right\rangle =o\big(h^{\frac12}\big).
\end{equation}
We study the symbol of the commutator using \eqref{eq.commutator} with semiclassical parameter $h^{\frac 12}$, and we deal with each term one by one. First of all, $s_{h,R}$ and $c_0$ belong to the symbol class $\mathscr{S}^0$ uniformly with respect to $h$. Thus,
\begin{equation}\label{eq.tests1}
 \left[\mathsf{Op}_{h^{\frac 12}}^B \left(s_{h,R} \right),\mathsf{Op}_{h^{\frac 12}}^B( h c_0(x))\right] = \mathsf{Op}_{h^{\frac 12}}^B(\sigma_{h,R}^1), \quad {\rm{with}} \quad \sigma_{h,R}^1 \in h^{\frac 32} \mathscr{S}^0.
\end{equation}
For the second term, note that $s_{h,R}$ is not uniformly in $\mathscr{S}^{-1}$, but $h^{\frac 12} s_{h,R}$ is. This follows from the facts that the support properties of $s_{h,R}$ ensures that $h^{\frac{1}{2}}\xi$ is bounded and that all the derivaties of $s_{h,R}$ are uniformly bounded. Since $c_1(x) \cdot \xi \in \mathscr{S}^1$ we deduce
\begin{equation}\label{eq.tests2}
 \left[\mathsf{Op}_{h^{\frac 12}}^B \left(s_{h,R} \right),\mathsf{Op}_{h^{\frac 12}}^B( h^{\frac{1}{2}} c_1(x) \cdot \xi )\right] = \mathsf{Op}_{h^{\frac 12}}^B(\sigma_{h,R}^2), 
 \end{equation}
 with
 $$\sigma_{h,R}^2 = \frac{h^{\frac 12}}{i} \lbrace s_{h,R}, h^{\frac 12} c_1(x) \cdot \xi \rbrace + h \mathscr{S}^0.
$$
Similarily, $h s_{h,R}$ is in the class $\mathscr{S}^{-2}$ and $|\xi|^2 \in \mathscr{S}^2$ so that
\begin{align} \nonumber
 &\left[\mathsf{Op}_{h^{\frac 12}}^B \left(s_{h,R} \right),\mathsf{Op}_{h^{\frac 12}}^B( |\xi|^2 )\right] = \mathsf{Op}_{h^{\frac 12}}^B(\sigma_{h,R}^3), \\ &\quad {\rm{with}} \quad \sigma_{h,R}^3 = \frac{h^{\frac12}}{i} \lbrace s_{h,R}, |\xi|^2  \rbrace + \frac{2\widehat{B}_0 h}{i}\big( \xi_1 \partial_{\xi_2} s_{h,R} - \xi_2 \partial_{\xi_1} s_{h,R}  \big),\label{eq.tests3}
\end{align}
where the remainder is $0$ thanks to the last item of Theorem~\ref{thm.comp} and due to the fact that $\|\xi\|^2$ is quadratic and independent of $x$. We insert the estimates \eqref{eq.tests1}, \eqref{eq.tests2} and \eqref{eq.tests3} in \eqref{eq.testup} together with the Calder\'on-Vaillancourt Theorem to deduce
\begin{multline}\label{eq.test+}
\Big\langle \mathsf{Op}_{h^{\frac 12}}^B \Big( \frac{h^{\frac12}}{i} \lbrace s_{h,R}, |\xi|^2  \rbrace + \frac{2\widehat{B}_0 h}{i}\big( \xi_1 \partial_{\xi_2} s_{h,R} - \xi_2 \partial_{\xi_1} s_{h,R}  \big) + \frac{h}{i} \lbrace s_{h,R}, c_1(x) \cdot \xi \rbrace \Big) u_h,u_h\Big\rangle \\
=o(h^{\frac 12}).
\end{multline}
We calculate the remaning symbols and find that
\begin{equation*}
\frac{2\widehat{B}_0 h^{\frac12}}{i}\big( \xi_1 \partial_{\xi_2} (h^{\frac12}s_{h,R}) - \xi_2 \partial_{\xi_1} (h^{\frac12}s_{h,R})  \big) + \frac{h^{\frac12}}{i} \lbrace h^{\frac12}s_{h,R}, c_1(x) \cdot \xi \rbrace = \mathscr{O}(h^{\frac 12} R^{-1}) \quad {\rm{in }} \, \mathscr{S}^0.
\end{equation*}
Therefore when we divide \eqref{eq.test+} by $h^{\frac 12}$ we find
\begin{equation}\label{eq.testup2}
\lim_{R \to \infty} \lim_{h \to 0} \Big\langle \mathsf{Op}_{h^{\frac 12}}^B \Big( \frac{1}{i} \lbrace s_{h,R}, |\xi|^2  \rbrace \Big) u_h,u_h\Big\rangle = 0.
\end{equation}
The symbol $s_{h,R}$ was chosen such that
\[\lbrace s_{h,R} , | \xi |^2 \rbrace = - \frac{4i\pi | k |}{L_\Lambda} e^{2i\pi k\cdot x} a(h^{\frac 12} \xi) \left(1-\chi\left(\frac{\xi\cdot\mathfrak{e}_\Lambda}{L_\Lambda R}\right)\right),  \]
and thus
\begin{multline*}
 \frac{ L_\Lambda}{4i\pi |k|}\mathsf{Op}_{h^{\frac 12}}^B \left(\lbrace s_{h,R} , | \xi |^2 \rbrace \right)= \mathsf{Op}_{h^{\frac 12}}^B \Big( e^{ikx} a(h^{\frac 12} \xi) \Big(1-\chi\Big(\frac{\xi\cdot\mathfrak{e}_\Lambda}{L_\Lambda R}\Big) \Big) \\= \Op \Big( e^{ikx} a(\xi) \Big(1-\chi_{Rh^{\frac 12},\Lambda}(\xi) \Big) \Big).
\end{multline*}
Therefore equation \eqref{eq.testup2} becomes
\[ \langle \mu^{\Lambda}, e^{ikx} a(\xi) \rangle = \lim_{R \to \infty} \lim_{h \to 0^+} \langle w_{h,R}^\Lambda, e^{ikx}a(\xi) \rangle = 0,\]
which is the expected result.
\end{proof}

\subsection{Analyzing the compact part}\label{sec.compactpart}

 In this section we prove that $\mu_\Lambda$ vanishes, under our geometric control assumption on the magnetic field. Recall from Remark~\ref{rem.Binfty} that $\mathcal{I}_\Lambda(B)$ is independent of $\xi$ and that $\mathcal{I}_\Lambda(B)(x)=B_\infty(x,\theta)$ when $\theta\in\Lambda^\perp\cap\mathbf{S}^1$. 
 
\begin{proposition}\label{prop.muloc}
Let $\mu$ be any accumulation point of $(w_h)_{h \to 0}$, and consider the associated measure $\mu_\Lambda$ given by Lemma \ref{l:magnetic-wigner-split}. If $\mathcal I_\Lambda(B) >0$ everywhere, then $\mu_\Lambda=0$.
\end{proposition}

To analyze this part of the measure and prove Proposition \ref{prop.muloc}, we can implement the two-microlocal techniques. The idea is to artificially add the one-dimensional variable $\eta = \frac{\xi\cdot\mathfrak{e}_{\Lambda}}{L_\Lambda Rh^{\frac12}}$, and to analyze the quantum limits for symbols $a(x,\xi,\eta)$, depending also on the new variable $\eta$. More precisely, we introduce the two-microlocal distribution:
$$
\mu_{h,R,\Lambda}: a\in\mathscr{C}^\infty_c(T^*\T^2\times \R)\mapsto \left\langle w_h,a\left(x,\xi,\frac{\xi\cdot\mathfrak{e}_{\Lambda}}{L_\Lambda h^{\frac12}}\right)\chi\left(\frac{\xi\cdot\mathfrak{e}_{\Lambda}}{L_\Lambda Rh^{\frac12}}\right)\right\rangle.
$$

Mimicking the proof of Lemma~\ref{l:magnetic-wigner-split}, we can verify that accumulation points of $(\mu_{h,R,\Lambda})$ satisfy the same invariance properties as $\mu$.

\begin{lemma}\label{l:magnetic-wigner-2microlocal}  Let $(u_h)_{h\rightarrow 0^+}$ be a sequence verifying~\eqref{e:semiclassical-quasimode}. Then the following holds:
\begin{enumerate}
 \item The corresponding sequence $(\mu_{h,R,\Lambda})_{h\rightarrow 0^+,R\rightarrow +\infty}$ is bounded in $\mathscr{D}'(T^*\T^2 \times \R)$. In particular, up to extraction, the sequence $(\mu_{h,R,\Lambda})_{h\rightarrow 0^+,R\rightarrow +\infty}$ converges in $\mathscr{D}'(T^*\T^2 \times \R)$ to some $\tilde{\mu}_\Lambda$ as $h\rightarrow 0^+$ and $R\rightarrow+\infty$ in this order.
 \item Any accumulation point $\tilde{\mu}_\Lambda$ (as $h\rightarrow 0^+$ and $R\rightarrow+\infty$ in this order) of $(\mu_{h,R,\Lambda})_{h\rightarrow 0^+,R\rightarrow +\infty}$ is a finite measure on $\T^2\times(\Lambda^{\perp}\cap\mathbf{S}^1)\times\R$,
which is invariant by the lifted geodesic flow
$$
\tilde{\varphi}^{t}(x,\xi,\eta)=(x+t\xi,\xi,\eta);
$$
\item $\mu_{\Lambda}(x,\xi)=\int_\R\tilde{\mu}_\Lambda(x,\xi,d\eta)$.
\end{enumerate} 
\end{lemma}
Observe from the invariance by the geodesic flow that \[\langle \tilde{\mu}_\Lambda,e^{ik\cdot x}a(\xi,\eta)\rangle =0, \quad k \notin \Lambda, \quad a\in\mathscr{C}^\infty_c(\R^2 \times \R).\] Therefore, we can replace $\tilde{\mu}_\Lambda$ by its average,
\begin{equation}\label{e:Fourier-coeff-2microlocal}
\tilde{\mu}_\Lambda=\mathcal{I}_\Lambda(\tilde{\mu}_\Lambda).
\end{equation}
Also, from the support properties of $\tilde{\mu}_\Lambda$, we can split this measure as
$$
\tilde{\mu}_\Lambda^\pm:=\tilde{\mu}_\Lambda|_{\T^2\times\left\{\pm\mathfrak{e}_\Lambda^\perp/L_\Lambda\right\}\times\R}.
$$

Our last step consists in determining an extra-regularity property for the measure $\tilde{\mu}_\Lambda$.

\begin{lemma}\label{l:invariance-2microlocal} Let $\tilde{\mu}_\Lambda$ be an accumulation point as in Lemma~\ref{l:magnetic-wigner-2microlocal}. Then, one has, for every $a\in\mathscr{C}^\infty_c(T^*\T^2\times\R)$,
$$
\left\langle \tilde{\mu}_\Lambda^{\pm},\eta \frac{\mathfrak{e}_\Lambda}{L_\Lambda}\cdot\partial_x\mathcal{I}_\Lambda(a) \pm \mathcal{I}_\Lambda(B)\partial_\eta \mathcal{I}_\Lambda(a) \right\rangle=0.
$$
\end{lemma}

Before proving Lemma \ref{l:invariance-2microlocal}, let us explain how to deduce Proposition \ref{prop.muloc}.

\begin{proof}[Proof of Proposition \ref{prop.muloc}] Lemma~\ref{l:magnetic-wigner-2microlocal} together with~\eqref{e:Fourier-coeff-2microlocal} ensure that $\tilde{\mu}_\Lambda$ is a measure on $\T^2\times\{\pm\mathfrak{e}_\Lambda^\perp/L_\Lambda\}\times\R$ with only Fourier coefficients along $\Lambda$. Hence, it can be identified with the two measures $\tilde \mu_\Lambda^\pm$ and we would like to prove that they are identically $0$. This follows from the facts that they are finite measures and that they are invariant by a flow all of whose trajectories escape to infinity. Indeed, thanks to Lemma~\ref{l:invariance-2microlocal}, these measures are invariant by one of the flows $\phi_{\Lambda,\pm}^t$ generated by
$$
X_{\Lambda,\pm}:=\eta \frac{\mathfrak{e}_\Lambda}{L_\Lambda}\cdot\partial_x \pm\mathcal{I}_\Lambda(B)\partial_\eta.
$$
All the trajectories of this flow escape to infinity, as soon as $\mathcal{I}_\Lambda(B)>0$ everywhere. Indeed, for every trajectory, one has $\eta(t)=\eta(0)\pm\int_{0}^t\mathcal{I}_\Lambda(B)(x(\tau))d\tau\rightarrow\pm\infty$. As $\tilde{\mu}_\Lambda^\pm$ is an invariant finite measure, it is necessarly $0$ as the mass in a compact part has to be preserved under the flow.
\end{proof}

It remains to prove Lemma \ref{l:invariance-2microlocal}, and this is where we will extract the subleading effect of the magnetic field. The first idea is again to use the eigenvalue equation in the form
\begin{equation}\label{eq.commutator17}
\left\langle \left[ \Op \left(a\, \chi_{Rh^{\frac 12},\Lambda} \right) , h^2 \mathscr{L}_\alpha \right] u_h,u_h \right\rangle = o(h^{\frac 32}),
\end{equation}
and to the get information from the subprincipal terms of this commutator. We recall the symbol of $h^2 \mathscr{L}_\alpha$, calculated in \eqref{e:symbol-laplacian},
\begin{equation}
h^2\mathscr{L}_\alpha =\Op\left(|\xi |^2+hc_1(x)\cdot\xi+h^2c_0(x)\right),
\end{equation}
where
\begin{equation}\label{e:expression-subprincipal-symbol}
c_1(x):=-2(A^{\text{per}}(x)+\alpha),\quad \text{and}\quad c_0(x):=|A^{\text{per}}(x)+\alpha|^2=\frac{|c_1(x)|^2}{4}.
\end{equation}
In equation \eqref{eq.commutator17}, we can use the invariance of $\tilde \mu_\Lambda^{\pm}$ to replace $a$ by its average along the geodesics, i.e. $a= \mathcal I_\Lambda(a)$. For such $b$'s the main term of the commutator will simplify, since the Poisson bracket was precisely encoding the invariance by the geodesic flow. In fact, there will be two terms of next order, of the form
$$
\left[\Op (a\, \chi), h^2\mathscr{L}\right]\simeq\frac{2h^{\frac{3}{2}}}{i}\Op\left( \left(\eta \frac{\mathfrak{e}_\Lambda}{L_\Lambda}\cdot\partial_xa +\frac{\xi\cdot\mathfrak{e}_{\Lambda}^{\perp}}{L_\Lambda}\mathcal{I}_\Lambda(B)\partial_\eta a \right) \chi \right).
$$
This estimate, together with \eqref{eq.commutator17}, directly gives the result of Lemma \ref{l:invariance-2microlocal}. Let us now proceed with the details of the proof. As before, we find useful to work with the rescaled quantization $\mathsf{Op}_{h^{\frac 12}}^B$ in order to deal with the singular symbol $\chi_{Rh^{\frac 12},\Lambda}$.

\begin{proof}[Proof of Lemma \ref{l:invariance-2microlocal}] Let $a\in\mathscr{C}^\infty_c(T^*\T^2\times\R)$ be such that $\mathcal{I}_\Lambda(a)=a$, \emph{i.e.} $a$ has only Fourier coefficients along $\Lambda$. We introduce the symbol $s_h$, which belongs to the symbol class $\mathscr{S}^0$ uniformly with respect to $h$,
\[s_h(x,\xi) = a\left(x,h^{\frac 12} \xi,\frac{\xi\cdot\mathfrak{e}_{\Lambda}}{L_\Lambda}\right)\chi\left(\frac{\xi\cdot\mathfrak{e}_{\Lambda}}{L_\Lambda R}\right), \]
and which satisfies
\[ \mathsf{Op}_{h^{\frac 12}}^B (s_h) = \Op \left(a\left(x,\xi,\frac{\xi\cdot\mathfrak{e}_{\Lambda}}{L_\Lambda h^{\frac12}}\right)\chi\left(\frac{\xi\cdot\mathfrak{e}_{\Lambda}}{L_\Lambda Rh^{\frac12}}\right)\right). \]
Again, we write the quasimode equation~\eqref{e:semiclassical-quasimode},
$$
\left\langle\left[ \mathsf{Op}_{h^{\frac 12}}^B (s_h) ,h^2\mathscr{L}_\alpha\right]u_h,u_h\right\rangle =o(h^{\frac32}).
$$
Recall the symbol of $\mathscr{L}_\alpha$,
\[ h^2\mathscr{L}_\alpha = h \mathsf{Op}_{h^{\frac 12}}^B( | \xi |^2 + h^{\frac{1}{2}} c_1(x) \cdot \xi + h c_0(x)),\]
from which we deduce
\begin{equation}\label{eq.testdown}
\left\langle \left[\mathsf{Op}_{h^{\frac 12}}^B \left(s_{h} \right),\mathsf{Op}_{h^{\frac 12}}^B(| \xi |^2 + h^{\frac{1}{2}} c_1(x) \cdot \xi + h c_0(x))\right]u_h,u_h\right\rangle =o\big(h^{\frac12}\big).
\end{equation}
As in the proof of Lemma \ref{lem.muup}, we check the contribution of each term one by one, using the composition rule Theorem \ref{thm.comp}. First, since $c_0(x)$ and $s_h$ belong to $\mathscr{S}^0$ uniformly with respect to $h$ we have
\begin{equation}\label{eq.testdown1}
 \left[\mathsf{Op}_{h^{\frac 12}}^B \left(s_h \right),\mathsf{Op}_{h^{\frac 12}}^B( h c_0(x))\right] = \mathsf{Op}_{h^{\frac 12}}^B(\sigma_{h}^1), \quad {\rm{with}} \quad \sigma_{h}^1 \in h^{\frac 32} \mathscr{S}^0.
\end{equation}
For the second term, note that $c_1(x) \cdot \xi \in \mathscr{S}^1$ and $h^{\frac 12} s_{h}$ is in $\mathscr{S}^{-1}$ uniformly with respect to $h$ so that
\begin{equation}\label{eq.testdown2}
 \left[\mathsf{Op}_{h^{\frac 12}}^B \left(s_h \right),\mathsf{Op}_{h^{\frac 12}}^B( h^{\frac{1}{2}} c_1(x) \cdot \xi )\right] = \mathsf{Op}_{h^{\frac 12}}^B(\sigma_{h}^2), \  {\rm{with}} \ \sigma_{h}^2 = \frac{h^{\frac 12}}{i} \lbrace s_{h}, h^{\frac 12} c_1(x) \cdot \xi \rbrace + h \mathscr{S}^0.
\end{equation}
Finally, the exact formula from Corollary~\ref{coro.comp.quadratic} gives
\begin{align} \nonumber
 &\left[\mathsf{Op}_{h^{\frac 12}}^B \left(s_{h} \right),\mathsf{Op}_{h^{\frac 12}}^B( |\xi|^2 )\right] = \mathsf{Op}_{h^{\frac 12}}^B(\sigma_{h}^3), \\ &\quad {\rm{with}} \quad \sigma_{h}^3 = \frac{h^{\frac12}}{i} \lbrace s_{h}, |\xi|^2  \rbrace + \frac{2\widehat{B}_0 h}{i}\big( \xi_1 \partial_{\xi_2} s_{h} - \xi_2 \partial_{\xi_1} s_{h}  \big). \label{eq.testdown3}
\end{align}
We insert the estimates \eqref{eq.testdown1}, \eqref{eq.testdown2} and \eqref{eq.testdown3} in \eqref{eq.testdown} together with the Calder\'on-Vaillancourt Theorem to deduce
\begin{multline}\label{eq.testdown+}
\Big\langle \mathsf{Op}_{h^{\frac 12}}^B \Big( \frac{h^{\frac12}}{i} \lbrace s_{h}, |\xi|^2  \rbrace + \frac{2\widehat{B}_0 h}{i}\big( \xi_1 \partial_{\xi_2} s_{h} - \xi_2 \partial_{\xi_1} s_{h}  \big) + \frac{h}{i} \lbrace s_{h}, c_1(x) \cdot \xi \rbrace \Big) u_h,u_h\Big\rangle \\
=o(h^{\frac 12}).
\end{multline}
Recalling that $a=\mathcal{I}_\Lambda(a)$, one has that
\begin{align*}
\left\{|\xi|^2,s_h \right\}&=2\xi \cdot\partial_xa\left(x,h^{\frac 12} \xi,\frac{\xi\cdot\mathfrak{e}_{\Lambda}}{L_\Lambda }\right)\chi\left(\frac{\xi\cdot\mathfrak{e}_{\Lambda}}{L_\Lambda R}\right)\\
&=2 \left(\frac{\xi\cdot\mathfrak{e}_{\Lambda}}{L_\Lambda }\right)\frac{\mathfrak{e}_\Lambda}{L_\Lambda}\cdot\partial_xa\left(x,h^{\frac 12} \xi,\frac{\xi\cdot\mathfrak{e}_{\Lambda}}{L_\Lambda }\right)\chi\left(\frac{\xi\cdot\mathfrak{e}_{\Lambda}}{L_\Lambda R}\right).
\end{align*}
Regarding the term involving the subsprincipal symbol $hc_1(x)\cdot\xi$, one finds
\begin{align*}
 h\left\{c_1(x)\cdot\xi,s_h\right\}&=h c_{1,\Lambda}(x)\frac{\mathfrak{e}_\Lambda}{L_\Lambda}\cdot\partial_xa\left(x,h^{\frac 12} \xi,\frac{\xi\cdot\mathfrak{e}_{\Lambda}}{L_\Lambda}\right)\chi\left(\frac{\xi\cdot\mathfrak{e}_{\Lambda}}{L_\Lambda R}\right)\\
 &\quad +h \frac{\xi\cdot\mathfrak{e}_{\Lambda}}{L_\Lambda}\left\{c_{1,\Lambda}(x),a\left(x,h^{\frac 12} \xi,\frac{\xi\cdot\mathfrak{e}_{\Lambda}}{L_\Lambda }\right)\chi\left(\frac{\xi\cdot\mathfrak{e}_{\Lambda}}{L_\Lambda R}\right)\right\}\\
&\quad +h \frac{\xi\cdot\mathfrak{e}_{\Lambda}^\perp}{L_\Lambda}\left\{c_{1,\Lambda}^\perp(x),a\left(x,h^{\frac 12} \xi,\frac{\xi\cdot\mathfrak{e}_{\Lambda}}{L_\Lambda}\right)\chi\left(\frac{\xi\cdot\mathfrak{e}_{\Lambda}}{L_\Lambda R}\right)\right\}.\\
 &= -h \frac{\xi \cdot \mathfrak{e}_{\Lambda}^\perp}{L_\Lambda} \partial_\eta a \Big(x,h^{\frac 12} \xi, \frac{\xi \cdot \mathfrak{e}_\Lambda}{L_\Lambda } \Big) \frac{\mathfrak{e}_\Lambda}{L_\Lambda} \cdot \partial_x c_{1,\Lambda}^\perp(x) \chi \left(\frac{\xi\cdot\mathfrak{e}_{\Lambda}}{L_\Lambda R}\right)\\&\quad + \mathscr{O}_R(h) + \mathscr{O}( h^{\frac 12} R^{-1}) \quad {\rm{in}} \quad \mathscr{S}^0,
\end{align*}
where we have set
$$
c_{1,\Lambda}(x):=\frac{\mathfrak{e}_\Lambda}{L_\Lambda}\cdot c_1(x),\quad\text{and}\quad c_{1,\Lambda}^{\perp}(x):=\frac{\mathfrak{e}_\Lambda^\perp}{L_\Lambda}\cdot c_1(x).
$$
Similarily, the magnetic term is
\begin{align*}
2\widehat{B}_0 h \big( \xi_1 \partial_{\xi_2} s_{h} - \xi_2 \partial_{\xi_1} s_{h}  \big) &= 2 \widehat{B}_0 h \frac{\xi \cdot e_\Lambda^\perp}{L_\Lambda} \partial_\eta a \Big(x,h^{\frac 12} \xi, \frac{\xi \cdot \mathfrak{e}_\Lambda}{L_\Lambda } \Big) \chi \left(\frac{\xi\cdot\mathfrak{e}_{\Lambda}}{L_\Lambda R}\right)\\ &\qquad + \mathscr{O}_R (h) + \mathscr{O}(h^{\frac 12} R^{-1}) \quad {\rm{in}} \quad \mathscr{S}^0.
\end{align*}
Therefore, equation \eqref{eq.testdown+} gives, after dividing by $h^{\frac 12}$ and using the Calder\'on-Vaillancourt Theorem,
\begin{multline*}
-2\left\langle\mathsf{Op}_{h^{\frac 12}}^B\left( \left(\frac{\xi\cdot\mathfrak{e}_{\Lambda}}{L_\Lambda }\right)\frac{\mathfrak{e}_\Lambda}{L_\Lambda}\cdot\partial_xa\left(x,h^{\frac 12} \xi,\frac{\xi\cdot\mathfrak{e}_{\Lambda}}{L_\Lambda }\right)\chi\left(\frac{\xi\cdot\mathfrak{e}_{\Lambda}}{L_\Lambda R}\right)\right)u_h,u_h\right\rangle\\
+\left\langle\mathsf{Op}_{h^{\frac 12}}^B\left(h^{\frac 12} \frac{\xi \cdot \mathfrak{e}_{\Lambda}^\perp}{L_\Lambda} \partial_\eta a \Big(x,h^{\frac 12} \xi, \frac{\xi \cdot \mathfrak{e}_\Lambda}{L_\Lambda } \Big) \frac{\mathfrak{e}_\Lambda}{L_\Lambda} \cdot \partial_x c_{1,\Lambda}^\perp(x) \chi \left(\frac{\xi\cdot\mathfrak{e}_{\Lambda}}{L_\Lambda R}\right)\right)u_h,u_h\right\rangle\\
-2\widehat{B}_0\left\langle\mathsf{Op}_{h^{\frac 12}}^B\left(h^{\frac 12} \frac{\xi \cdot e_\Lambda^\perp}{L_\Lambda} \partial_\eta a \Big(x,h^{\frac 12} \xi, \frac{\xi \cdot \mathfrak{e}_\Lambda}{L_\Lambda } \Big) \chi \left(\frac{\xi\cdot\mathfrak{e}_{\Lambda}}{L_\Lambda R}\right)\right)u_h,u_h\right\rangle
\\
=o_R(1)+\mathscr{O}(R^{-1}).
\end{multline*}
We then rescale back to the $h$-quantization,
\begin{multline*}
-2\left\langle\Op\left(\left(\frac{\xi\cdot\mathfrak{e}_{\Lambda}}{L_\Lambda h^{\frac12}}\right)\frac{\mathfrak{e}_\Lambda}{L_\Lambda}\cdot\partial_xa\left(x,\xi,\frac{\xi\cdot\mathfrak{e}_{\Lambda}}{L_\Lambda h^{\frac12}}\right)\chi\left(\frac{\xi\cdot\mathfrak{e}_{\Lambda}}{L_\Lambda Rh^{\frac12}}\right)\right)u_h,u_h\right\rangle\\
+\left\langle\Op\left(\frac{\xi\cdot\mathfrak{e}_{\Lambda}^\perp}{L_\Lambda}\partial_\eta a\left(x,\xi,\frac{\xi\cdot\mathfrak{e}_{\Lambda}}{L_\Lambda h^{\frac12}}\right)\frac{\mathfrak{e}_{\Lambda}}{L_\Lambda}\cdot \partial_xc_{1,\Lambda}^\perp(x)\chi\left(\frac{\xi\cdot\mathfrak{e}_{\Lambda}}{L_\Lambda Rh^{\frac12}}\right)\right)u_h,u_h\right\rangle\\
-2\widehat{B}_0\left\langle\Op\left(\frac{\xi\cdot\mathfrak{e}_{\Lambda}^{\perp}}{L_\Lambda}\partial_\eta a\left(x,\xi,\frac{\xi\cdot\mathfrak{e}_{\Lambda}}{L_\Lambda h^{\frac12}}\right)\chi\left(\frac{\xi\cdot\mathfrak{e}_{\Lambda}}{L_\Lambda Rh^{\frac12}}\right)\right)u_h,u_h\right\rangle
\\
=o_R(1)+\mathscr{O}(R^{-1}).
\end{multline*}
Hence, letting $h\rightarrow 0^+$ and $R \rightarrow+\infty$ (in this order), one finds that, for every $a\in\mathscr{C}^\infty_c(T^*\T^2\times \R)$ verifying $\mathcal{I}_\Lambda(a)=a$,
$$
\left\langle \tilde{\mu}_\Lambda,-2\eta \frac{\mathfrak{e}_\Lambda}{L_\Lambda}\cdot\partial_xa +\frac{\xi\cdot\mathfrak{e}_{\Lambda}^{\perp}}{L_\Lambda}\left(\frac{\mathfrak{e}_{\Lambda}}{L_\Lambda}\cdot \partial_xc_{1,\Lambda}^\perp-2\widehat{B}_0\right)\partial_\eta a \right\rangle=0,
$$
where we recall that $\xi\cdot\mathfrak{e}_{\Lambda}^{\perp}=\pm L_\Lambda$ on the support of $\tilde{\mu}_\Lambda$. Using~\eqref{e:expression-subprincipal-symbol} together with the invariance by the geodesic flow and the fact that $\tilde{\mu}_\Lambda$ is carried by $\T^2\times(\Lambda^\perp\cap\mathbf{S}^1)\times\R$, one finds the expected result
$$
\left\langle \tilde{\mu}_\Lambda,\eta \frac{\mathfrak{e}_\Lambda}{L_\Lambda}\cdot\partial_xa +\frac{\xi\cdot\mathfrak{e}_{\Lambda}^{\perp}}{L_\Lambda}\mathcal{I}_\Lambda(B)\partial_\eta a \right\rangle=0.
$$
\end{proof}

\subsection{Conclusion}

 Gathering all these results, we are in position to prove the following Theorem.
\begin{theorem}\label{t:maintheo-configurationspace} Suppose that, for every $\Lambda$ in $\mathcal{L}_1$, $\mathcal{I}_\Lambda(B)>0$ everywhere. Then, one has 
 $$
 \mathcal{N}(\mathscr{L}_\alpha)=\left\{\operatorname{Leb}\right\}.
 $$
More precisely, any $\mu(\dd x,\dd\xi)\in\mathcal{M}(\mathscr{L}_\alpha)$ is of the form $
\dd x\,\nu_0(\dd\xi),$
where $\nu_0$ is a probability measure on $\mathbf{S}^1$.
\end{theorem}

\begin{proof} Let $\nu$ be an element in $\mathcal{N}(\mathscr{L}_\alpha)$. Denote by $\mu$ a semiclassical measure associated to the sequence of eigenfunctions $(u_h)_{h\rightarrow 0^+}$ used to generate $\nu$. According to Lemma~\ref{l:anantharamanmacia}, it is sufficient to analyze $\mu|_{\T^2\times(\Lambda^\perp\cap\mathbf{S}^1)}$ in order to conclude. Thanks to~\eqref{e:splitting-2microlocal-measure}, one has that
$$
\mu|_{\T^2\times(\Lambda^\perp\cap\mathbf{S}^1)}=\mu^{\Lambda}|_{\T^2\times(\Lambda^\perp\cap\mathbf{S}^1)}+\mu_{\Lambda}|_{\T^2\times(\Lambda^\perp\cap\mathbf{S}^1)}.
$$
 According to Proposition \ref{prop.muloc}, our assumption on $B$ ensures that $\mu_{\Lambda}\equiv 0$. Finally,~\eqref{e:splitting-2microlocal-measure3} allows to conclude that $\mu^{\Lambda}|_{\T^2\times(\Lambda^\perp\cap\mathbf{S}^1)}$ has no nonzero Fourier coefficients which ends the proof of the theorem.
\end{proof}

\section{Proof for variable magnetic fields}\label{s:proof}

 In this final section, we consider slightly better quasimodes, namely we suppose that $(u_{h_n})_{n\geqslant 1}$ is a sequence in $H^2(\T^2,L)$ such that
\begin{equation}\label{e:semiclassical-quasimode-strong}
 h_n^2(\mathscr{L}_\alpha +V) u_{h_n}=u_{h_n}+o_{L^2}\left(h_n^2\right),\quad \|u_{h_n}\|_{L^2(\T^2)}=1,\quad h_n\rightarrow 0^+.
\end{equation}
In particular, $(u_{h_n})_{n\geqslant 1}$ verifies~\eqref{e:semiclassical-quasimode} and, up to extraction, we denote by $\mu$ a semiclassical measure of this sequence of quasimodes. Thanks to Lemma~\ref{l:magnetic-wigner} and to Theorem~\ref{t:maintheo-configurationspace}, one knows that $\mu(\dd x,\dd \xi)$ is a probability measure on $S^*\T^2$ that is of the form $\dd x\nu_0(\dd\xi)$ with $\nu_0$ a probablity measure on $\mathbf{S}^1$. Our final step to prove Theorem~\ref{t:maintheo} consists in showing that $\nu_0$ is the Lebesgue measure along the circle $\mathbf{S}^1$. To see this, it is sufficient to pick a smooth test function $a(\xi)\in\mathscr{C}^\infty_c(\R^2)$ and to write the quasimode equation~\eqref{e:semiclassical-quasimode-strong} one last time:
$$
\left\langle \left[\Op(a),h^2(\mathscr{L}_\alpha +V)\right]u_h,u_h\right\rangle=o(h^2).
$$
Using the composition Theorem~\ref{thm.comp} together with the Calder\'on-Vaillancourt Theorem~\ref{thm.cv}, one finds that
$$
\left\langle \Op\left(\left\{a(\xi),c_1(x)\cdot\xi \right\}-2\widehat{B}_0(\xi_1\partial_{\xi_2}-\xi_2\partial_{\xi_1})a\right)u_h,u_h\right\rangle=o(1).
$$
\begin{remark}
 Note that we used here that $[\Op(a),V]=o(1)$ as $h\rightarrow 0^+$. This is the only place of the article where we use the regularity of $V$ (through our use of pseudodifferential calculus).
\end{remark}
Letting $h$ to $0$ and using~\eqref{e:expression-subprincipal-symbol}, one finds that
$$
\left\langle \mu,\xi_1\{a(\xi),A_1^{\text{per}}(x)\}+\xi_2\{a(\xi),A_2^{\text{per}}(x)\} +\widehat{B}_0(\xi_1\partial_{\xi_2}-\xi_2\partial_{\xi_1})a\right\rangle =0.
$$
In fact, most of the terms in this sum cancel as the measure $\mu$ is of the form $\dd x\nu_0(\dd\xi)$:
$$
\left\langle \mu,\xi_j\{a(\xi),A_j^{\text{per}}(x)\}\right\rangle =\int_{\mathbf{S}^1}\xi_j\partial_\xi a\cdot\left(\int_{\T^2}\partial_xA_j^{\text{per}}\dd x\right)\nu_0(\dd\xi)=0.
$$
From this, we get that
$$
\left\langle \nu_0,\widehat{B}_0(\xi_1\partial_{\xi_2}-\xi_2\partial_{\xi_1})a\right\rangle =0.
$$
As $\widehat{B}_0\neq 0$, we can conclude that $\nu_0$ is the Lebesgue measure along the circle. This concludes the proof of Theorem~\ref{t:maintheo}.

\appendix
\section{Results for more general magnetic fields}
\label{s:general}

 Recall that Theorem~\ref{t:maintheo} depends on the fact that
$$
B_\infty(x,\theta)=\lim_{T\rightarrow +\infty}\frac{1}{T} \int_0^T B(x+t\theta)\dd t,\quad (x,\theta)\in S^*\T^2,
$$
is a positive function. When this geometric control condition is not satisfied, we can still conclude some regularity properties of the elements in $\mathcal{N}(\mathscr{L}_\alpha)$ following the lines of~\cite{MaciaRiviere2018}. More precisely, we can define a set of critical geodesics along $\Lambda$ (with respect to $B$) as
$$
\mathcal{Z}_\Lambda(B):=\left\{x_0\in\T^2: \forall \theta\in \mathbf{S}^1\cap\Lambda^\perp\ \text{such that}\ B_\infty(x_0,\theta)=0\right\}.
$$
This is a set of closed geodesics in the direction $\Lambda^\perp$ that could be either finite or infinite. We denote then by $\mathcal{N}_\Lambda(B)$ the convex closure of the measures $c\delta_\gamma$, where $c\in[0,1]$ and where $\delta_\gamma$ is the normalized Lebesgue measure along a closed geodesic $\gamma\in\mathcal{Z}_\Lambda(B)$. As a result of our analysis, we can prove the following theorem which is an analogue for magnetic fields of~\cite[Th.~1.1]{MaciaRiviere2018}.

\begin{theorem}\label{t:generaltheorem} Suppose that $\widehat{B}_0>0$. Then, one has:
\begin{enumerate}
 \item If $B_\infty\geq 0$, then any $\nu\in\mathcal{N}(\mathscr{L}_\alpha)$ can be decomposed as
 $$\nu=c_0+\sum_{\Lambda\in\mathcal{L}_1}\nu_{\operatorname{sing}}^\Lambda,
 $$
 where $c_0\in[0,1]$ and $\nu_{\operatorname{sing}}^\Lambda\in \mathcal{N}_\Lambda(B)$.
 \item Without any assumption on $B_\infty$, one has, for any $\nu\in\mathcal{N}(\mathscr{L}_\alpha)$ and for any closed geodesic $\gamma$,
 $$
 \nu(\gamma)>0\quad\Longrightarrow\quad \exists\Lambda\in\mathcal{L}_1\ \text{such that}\ \gamma\subset \mathcal{Z}_\Lambda(B)
 $$
 \item Without any assumption on $B_\infty$ and if $k\neq 0$, any $\nu\in\mathcal{N}(\mathscr{L}_\alpha)$ verifies
 $$
 \widehat{\nu}_k\neq 0\quad\Longrightarrow\quad \exists \Lambda\in\mathcal{L}_1\ \text{such that}\ k\in\Lambda\ \text{and}\ \mathcal{Z}_\Lambda(B)\neq\emptyset.
 $$
\end{enumerate}
\end{theorem}

As an illustration of the first item, if $B\geq 0$ and if there exists only a single geodesic (issued from the point $(x_0,\pm\theta_0)$) such that $B_\infty(x_0,\pm\theta_0)=0$, then any $\nu\in\mathcal{N}(\mathscr{L}_\alpha)$ is of the form $c_0\dd x+(1-c_0)\delta_{\gamma_0}$, where $\gamma_0$ is the closed geodesic issued from $(x_0,\theta_0)\in S^*\T^2$.

In view of understanding the scope of this theorem, it is worth noting that 
$$
\widehat{B}_0>0\quad\Longrightarrow\quad \sharp\left\{\Lambda\in\mathcal{L}_1:\ \mathcal{Z}_\Lambda(B)\neq\emptyset\right\}<\infty.
$$
Indeed, recalling Remark~\ref{rem.Binfty}, one can observe that, for every $ x\in \T^2$ and for every $\varepsilon>0$,
$$
\mathcal{I}_\Lambda(B)(x)\geq \widehat{B}_0-\sum_{k\in\Z\setminus 0}\left|\widehat{B}_{k\mathfrak{e}_\Lambda}\right|\geq \widehat{B}_0-\frac{\|B-\widehat{B}_0\|_{\dot{H}^{\frac{1}{2}+\varepsilon}(\T^2)}}{L_\Lambda^{\frac{1}{2}+2\varepsilon}}\left(\sum_{k\neq 0}\frac{1}{k^{1+2\varepsilon}}\right)^{\frac{1}{2}},
$$
from which we can deduce the above property as $\sharp\{\Lambda: L_\Lambda<R\}<\infty$ for every $R>0$. In particular, as soon as $\widehat{B}_0>0$, any $\nu\in\mathcal{N}(\mathscr{L}_\alpha)$ has its Fourier coefficients along a finite number of primitive rank $1$ sublattices of $\Z^2$ (that depend only on $B$). 

\begin{proof} Let $\nu$ be an element in $\mathcal{N}(\mathscr{L}_\alpha)$. Up to another extraction, we can suppose that the sequence of quasimodes used to generate $\nu$ gives rise to an unique semiclassical measure $\mu$ as in Lemma~\ref{l:magnetic-wigner}. According to this Lemma and to Lemma~\ref{l:anantharamanmacia}, one has $\nu=\pi_*(\mu)$ and 
$$
\mu(x,\xi)=\widehat{\mu}_0(\xi)|_{\T^2\times(\Omega_2\cap\mathbf{S}^1)}+\sum_{\Lambda\in\mathcal{L}_1} \nu_\Lambda^\pm(x)\otimes\delta\left(\xi\mp\frac{\mathfrak{e}_\Lambda^\perp}{L_\Lambda}\right),
$$
where $\nu_\Lambda^\pm(x)$ has only Fourier coefficients along $\Lambda$. Thanks to Lemma~\ref{l:magnetic-wigner-split}, $\mu$ can be further decomposed according to $\mu_\Lambda$ and $\mu^{\Lambda}$. Then, Lemma~\ref{lem.muup} (together with another application of Lemma~\ref{l:anantharamanmacia}) ensures that the contribution of $\mu^{\Lambda}|_{\T^2\times\Lambda^\perp}$ has no nonzero Fourier coefficients. Hence, if we denote by $\tilde{\nu}_{\Lambda}^\pm$ the contribution of $\mu_\Lambda$, one finds that 
$$
\nu(\dd x)=c_0\dd x+\sum_{\Lambda\in\mathcal{L}_1}\left(\tilde{\nu}_\Lambda^+(\dd x)+\tilde{\nu}_\Lambda^-(\dd x)\right),
$$
where $c_0\in[0,1]$ is a constant. Moreover, $\tilde{\nu}_\Lambda^\pm$ has only Fourier coefficients along $\Lambda$ and, thanks to Lemma~\ref{l:invariance-2microlocal}, it is the pushforward of a finite measure\footnote{Compared with Lemma~\ref{l:invariance-2microlocal}, we make the small abuse of notations to remove the $\xi$-variable as the measure is carried along $\pm\mathfrak{e}_\Lambda/L_\Lambda$ along this variable.} $\tilde{\mu}_\Lambda^\pm(x,\eta)$ on $\T^2\times \R$ which has only Fourier coefficients along $\Lambda$ and which is invariant by the vector field
$$
X_{\Lambda,\pm}=\eta\frac{\mathfrak{e}_\Lambda}{L_\Lambda}\cdot\partial_x\pm\mathcal{I}_\Lambda(B)\partial_\eta.
$$
Equivalently, $\tilde{\mu}_\Lambda^\pm(x,\eta)$ can be identified with a measure $\rho_{\Lambda}^\pm(y,\eta)$ on $T^*\T_\Lambda$, with $\T_\Lambda=\langle\Lambda\rangle/\Lambda$, which is invariant by the vector field $Y_\Lambda^\pm(y,\eta)=\eta\partial_y\pm\mathcal{I}_\Lambda(B)(y)\partial_\eta
$. Hence, we can derive the regularity properties of the measure $\nu$ from the fact that each measure $\tilde{\nu}_\Lambda^\pm$ is the pushforward of an invariant measure.

In the case where $B_\infty\geq 0$ everywhere, one has that, for every $\Lambda$, $\mathcal{I}_\Lambda(B)\geq 0$. Moreover, the function $\mathcal{I}_\Lambda(B)$ is not identically $0$ as $\widehat{B}_0>0$. In particular, all the trajectories of the flow generated by $Y_\Lambda^\pm$ escape to infinity or converge to a point in $\{(y,0):\mathcal{I}_\Lambda(B)(y)=0\}$. Hence, any invariant probability measure by this vector field is in the (closed) convex hull of the Dirac masses carried by this critical set of points. As $\tilde{\nu}_\Lambda^\pm$ is the pushforward of a finite measure invariant by this vector field, this concludes the first part of the proof of the theorem.

Let us now discuss the general case where $\mathcal{I}_\Lambda(B)$ may change sign (meaning the last two items of Theorem~\ref{t:generaltheorem}). For the last item, it is a direct consequence of Proposition \ref{prop.muloc}. Indeed, if $\widehat{\nu}_k\neq 0$, it means that the measure $\mu_\Lambda$ is not identically $0$ where $\Lambda$ is the rank $1$ lattice containing $k$. According to Proposition~\ref{prop.muloc}, one has that $\mathcal{I}_\Lambda(B)$ is not positive everywhere.
In order to prove the second item, suppose that there exists a closed geodesic $\gamma_0$ such that $\nu(\gamma_0)>0$ and denote by $\Lambda_0$ the rank $1$ sublattice such that $\Lambda_0^\perp$ is parallel to $\gamma_0$. One can pick a smooth function $\chi_{\gamma_0}\in\mathscr{C}^\infty(\T^2,[0,1])$ having all its Fourier coefficients in $\Lambda_0$, equal to $1$ in an $\epsilon$-neighborhood of $\gamma_0$ and $0$ outside a $2\epsilon$-neighborhood. One has $\nu(\chi_{\gamma_0})\geq \nu(\gamma_0)$. Observe now that, as $\epsilon$ goes to $0$, $\int_{\mathbb{T}^2}\chi_{\gamma_0}\dd x$ tends to $0$ by the dominated convergence Theorem. Hence, taking $\epsilon$ small enough, one has 
$$
0<\nu(\gamma_0)\leq  \widehat{\nu}_0\int_{\mathbb{T}^2}\chi_{\gamma_0}\dd x +\sum_{k\neq 0}\widehat{\nu}_k\widehat{\chi}_{\gamma_0}(-k)\leq \frac{1}{2}\nu(\gamma_0)+\sum_{k\neq 0}\widehat{\nu}_k\widehat{\chi}_{\gamma_0}(-k).
$$
Hence, there exists some $k\in\Lambda_0\setminus\{0\}$ such that $\widehat{\nu}_k\neq 0$. In particular, the second item is a consequence of the third one.
\end{proof}

\bibliographystyle{plain}
\bibliography{refs}

\end{document}